\documentclass{amsart}
\usepackage{latexsym,amssymb,amsfonts,amsmath, amscd, euscript, MnSymbol}

\usepackage{pdfsync}
\usepackage{enumerate}

\usepackage[pdftex]{graphicx}
\usepackage{wrapfig}

\usepackage{epsfig}
\usepackage{epstopdf}
\usepackage{enumerate}
\usepackage{verbatim}
\usepackage{float}

\def\N{\mathbb N}
\def\R{\mathbb{R}}

\newtheorem{theorem}{Theorem}

\newtheorem{corollary}[theorem]{Corollary}

\newtheorem{definition}[theorem]{Definition}
\newtheorem{example}[theorem]{Example}
\newtheorem{lemma}{Lemma}
\newtheorem{fact}{Fact}
\newtheorem{proposition}{Proposition}

\newtheorem*{question}{Question}

\DeclareMathOperator{\gal}{Gal}
\DeclareMathOperator{\Min}{Min}
\DeclareMathOperator{\Sup}{Sup}

\title[Injective envelopes and Ferrers languages]{Injective envelopes of transition systems and Ferrers languages \dagger}
\author[M.Kabil]{Mustapha Kabil}
\address{Laboratoire Math\'ematiques et Applications, D\'epartement de Math{\'e}matiques, Facult\'e des Sciences et Techniques, Universit\'e Hassan II -Casablanca, BP 146 Mohammedia, Morocco. }
\email{kabilfstm@gmail.com}

\author [M. Pouzet]{Maurice Pouzet}
\address{Univ. Lyon, Universit\'e Claude-Bernard Lyon1, CNRS UMR 5208, Institut Camille Jordan,  43 bd. 11 Novembre 1918, 69622 Villeurbanne Cedex, France and Mathematics \& Statistics Department, University of Calgary, Calgary, Alberta, Canada T2N 1N4}
\email{pouzet@univ-lyon1.fr, mpouzet@ucalgary.ca }

\date{\today}


\begin{document}

\maketitle

 \dedicatory {\dagger \; Dedicated to the memory of Maurice Nivat}

\begin{abstract} We consider reflexive and involutive  transition systems over an ordered alphabet $A$ equipped with an involution. We give a description of the injective envelope of any two-element 
set in terms of Galois lattice, from which we derive a test of its
finiteness. Our description leads to the notion of Ferrers language.
\end{abstract}

\section{Introduction and presentation of the main results}

This paper is about involutive and reflexive  transition sytems from a metric point of view.\\
 This point of view, inspired from the work of Quilliot  (1983), and applied first to posets and graphs, was initiated by the second author \cite{pouzet} and developped throught the theses of Jawhari (1983),  Misane (1984) and several papers \cite{JaMiPo} (1986), \cite{pouzet-rosenberg} (1994), \cite {saidane} (1992), \cite {kabil-pouzet} (1998),  \cite {kabil-pouzet-rosenberg} (2018). It  consists to view arbitrary  transition systems as metric spaces. The distance between two states is a language instead of a non-negative real.   To a transition system   $M:= (Q, T)$, with set of states $Q$ and set of transitions $T$ over an alphabet $A$ we associate a map  $d_M$ from $Q\times Q$ into the set $\powerset({A^{\ast}})$ of languages over $A$. The value  $d_M(x,y)$ is the language accepted by the automaton $\mathcal A:= (M, \{x\}, \{y\})$ having $x$ as an initial state and  $y$ as a final state. The set $\powerset({A^{\ast}})$  is an ordered monoid, the monoid operation being the concatenation of languages (with neutral element  $\{\Box\}$, the language reduced to the empty word $\Box$) and the order the reverse of inclusion. The map $d_M$ has  similar properties of an ordinary distance (e.g. it satisfies the triangular inequality). Hence, we may  use concepts and techniques of the theory of metric spaces in the study of transition systems as well as classes of transition systems. Concepts  of  \emph{balls},  \emph{hyperconvex metric space} and 
 \emph{non-expansive}  maps between metric spaces extend to transition systems and more generally to metric spaces over $\powerset({A^{\ast}})$. Due to the fact that joins  exist in the set of values, the category of metric spaces with the non-expansive maps as morphisms has products. Then, one may also define \emph{retractions} and \emph{coretractions}, and by considering \emph{isometries} as approximations of coretractions, \emph{injective} metric spaces and \emph{absolute retracts}. In ordinary metric spaces, the distance is symmetric. To be closer to this situation, it is convenient to suppose that the value of $d_M(x,y)$ determines the value of $d_M(y,x)$; for that, we suppose that the alphabet is equipped with an involution $-$ and our transition systems  $M$ are \emph{involutive}, in the sense that $(x, \alpha, y)\in T$ if and only if $(y, \alpha, x)\in T$. Once the involution is extended to $A^{\ast}$ and then to $\powerset({A^{\ast}})$, we have $d_M(x,y)= \overline {d_M(y,x)}$. Then, one can extend the definion  of metric spaces to transition systems in a natural way, see \cite{pouzet-rosenberg}.  \\

This work is a continuation of the work published in 
\cite{kabil, kabil-pouzet,  kabil-pouzet-rosenberg}.  We require that transition systems $M$ are  \emph{reflexive}, that is  every letter occurs to every vertex:  $(x, \alpha, x)\in T$ for every $x\in Q$ and $\alpha \in A$. In this case,  distances values are final segments of $A^{\ast}$ equipped with the subword ordering, that is subsets $F$ of $A^{*}$ such that $u\in F$ and $u\leq v$ for the subword ordering imply $v\in F$. It turns out that several properties of involutive and reflexive systems and more generally metric spaces over the set $\mathbf F(A^{\ast})$ of final segments of $A^{*}$  rely almost uniquely on the structure of  $\mathbf F(A^{\ast})$. According to the terminology of Kaarli and Radeleczki \cite{kaarli-radeleczki}, this structure is the dual of   an \emph{ integral involutive quantale} ($I^2Q$ for short); here we stick to the name of  \emph{Heyting algebra} that we used in a series of papers.   This is a complete lattice $\mathcal H$ with a monoid
operation 
 (not necessarily commutative) and an  involution $-$ which is isotone and reverses the operation. In order to be closer to the operation of concatenation of languages  we denote by $\cdot$ the monoid operation. We suppose that the neutral element of the monoid, that we denote $1$,  is the least element  of the ordering and we suppose that  the distributivity law below holds
\begin{equation}\label{eq1}
 \bigwedge \{\mathit{p}_{\alpha } \cdot \mathit{q} : \alpha \in I\} =\bigwedge \{ \mathit{p}_{\alpha }: \alpha \in I\}\cdot\mathit{q}.  
 \end{equation}

As shown in  \cite{JaMiPo},  the notions of injective, absolute retract  and hyperconvex spaces over a Heyting algebra coincide, this being essentially due to the fact that the set of  the values of the distance,  being an Heyting algebra,  can be equipped with a distance and that every metric space can be embedded into a power of that metric space. Furthermore, every space has an injective envelope.
 
In particular, every metric space over  
 the Heyting algebra  $\mathbf F(A^{\ast})$ has an injective envelope. The study of such  injective envelope was initiated in  \cite{kabil-pouzet}.  
 It is  based on  the properties of the 
injective envelope of two-element metric spaces. A large account  of its properties was given in  \cite{kabil} and  \cite {kabil-pouzet-rosenberg}.  In this paper, we look at  the many facets of this object which have not been published yet. 

If $F$ is a final segment of $A^{\ast }$,
the injective envelope $\mathcal S_{F}$ of the two-element space $\left\{
x,y\right\} $ such that $d(x,y)=F$ is associated \ to an involutive and reflexive 
transition system. Let $M_{F}$ be this transition system and $\mathcal{A}_{F}$ be the automaton $(M_{F},\left\{ x\right\}, \left\{
y\right\} ) $. 

We characterize first this injective envelope in terms of reflexive and involutive transition systems. 
\begin{theorem}\label{thm:main1}
Let $\mathcal{A}:=(M,\left\{ x\right\}
,\left\{ y\right\} )$ be a  reflexive and involutive  automaton accepting a final segment  $F$ of $A^{\ast}$.  Then $\mathcal{A}$ is isomorphic to  $\mathcal{A}_F$ iff for every reflexive and involutive automaton $\mathcal{A}^{\prime
}:=(M^{\prime },\left\{ x^{\prime }\right\}
,\left\{ y^{\prime }\right\} )$ which accepts $F$, the following properties hold:
\begin{enumerate}[{(i)}]
\item Every automata morphism $f:\mathcal{A}\rightarrow \mathcal{A}^{\prime
}$, if any,  is an isometric embedding; 
\item The map $g:\{x',y'\}\rightarrow M$ such that $g(x')=x$ and $g(y')=y$ extends to a morphism of automata from $\mathcal A'$ to $\mathcal A$.
\end{enumerate}\end{theorem}
We will obtain this result as a consequence of a characterization of the injective envelope among generalized metric spaces (Theorem \ref{thm:main2}) given in Section \ref{section metric}. 

Then,  we develop an approach in terms of Galois correspondence. 

To a final segment $F$ of $A^{*}$ we associate the incidence structure $R:=(A^{*},\rho_F,A^{*})$ where
 $\rho $ is the binary relation on $A^{*}$ defined by
$u  \rho_F  v$ if the concatenation $u v$ of $u$ and $v$ belongs to $F$. For each $v\in A^{*}$, let
$R^{-1}(v):=\left\{ u\in A^{*}\text{ : }u\rho_F
v\right\} $ and for $V \subseteq A^{*}$, let $R_{\wedge} ^{-1}(V ):= \underset {v\in V}{
\bigcap }\rho_F^{-1}(v)$. The collection of all $R_{\wedge}^{-1}(V )$ for $
V \subseteq A^{*}$, once ordered by inclusion, forms a complete lattice $\gal(R)$,
called the \textit{Galois lattice} associated to the incidence structure $R$.
As a subset of $\mathcal H:=\mathbf F(A^{\ast })$, it inherits the metric structure $d_{\mathcal H}$ of $
\mathcal H$. Containing $x:=A^{\ast }$ and $y:=F$, two elements which
verify $d_{\mathcal H}(x, y)=F$, this metric space is the injective envelope of $
\left\{ x,y\right\} $. 

For  concrete examples, suppose  
that $F$ is a finite union of final segments $
F_{0},\dots  F_{i}, \dots F_{k-1}$  and that each $F_{i}$ is generated by $X_{i}$, a set of
words $u_{i}$ of the same length $n_{i}$, all of the from $u_{i}:= 
a_{0}^{i}$ $a_{1}^{i}\cdots a_{n_{i}-1}^{i}$ with $a_{j}^{i}\in
X_{i_{j}}\subseteq A$. Let $\mathbf{n}_{0}\otimes \cdots \otimes \mathbf{n}%
_{k-1}$ be the direct product of $k$ chains $\mathbf{n}_{0}$, ... $\mathbf{n}%
_{k-1}$ where $\mathbf{n}_{i}:=\left\{ 0,...,n_{i}-1\right\} $ is equipped
with the natural ordering, and let $\mathbf F(\mathbf{n}_{0}\otimes \cdots \otimes
\mathbf{n}_{k-1})$ be the collection of final segments of $\mathbf{n}%
_{0}\otimes\cdots \otimes \mathbf{n}_{k-1}$ ordered by inclusion.

We prove (see Section 5):

\begin{theorem}\label{injenv}
 As a lattice, the injective envelope $\mathcal S_{F}$ can be
identified with an intersection closed subset  of the set
$\mathbf F(\mathbf{n}_{0}\otimes \cdots \otimes \mathbf{n}_{k-1})$. Moreover
if $\downarrow $ $u_{i}\cap $ $\downarrow u_{j}$
= $\left\{ \Box \right\} $ whenever $u_{i}\in $ $X_{i}$, $u_{j}\in $ $
X_{j}$, and $i\neq j$, then $\mathcal S_{F}$ identifies to the full set $\mathbf F(\mathbf{n}
_{0}\otimes \cdots \otimes \mathbf{n}_{k-1})$.
\end{theorem}

We recall that a poset $P$ is {\it well-quasi-ordered}, in brief w.q.o., if it well-founded (every non empty subset contains a minimal element) and  contains no infinite antichain. A fundamental result of G.Higman \cite{higman} asserts that the free ordered monoid $A^{\ast}$ is w.q.o. whenever the alphabet $A$ is w.q.o.. The  set of final segments of a w.q.o. set, once ordered by reverse of the inclusion, is  well-founded \cite {higman}.  Hence, if our alphabet  $A$ is w.q.o., every final segment $F$ of $A^{*}$  is generated by finitely many
words $u_{0}$, ..., $u_{k-1}$, hence has the form mentioned above.
Consequently, the corresponding injective envelope is finite. Concerning its
size, let us mention that if $k=2$, then $\mathbf F(\mathbf{n}_{0}\otimes \mathbf{n}%
_{1})$ has size $\left(
\begin{array}{c}
n_{0}\cdot n_{1} \\
n_{0}
\end{array}
\right) \left[ 1\right] $ $\left( \text{3.16, p. 80}\right) .$ If $n_{0}$ =
... = $n_{k-1}$ = 2, then $(\mathbf F(\mathbf{n}_{0}\otimes \cdots\otimes \mathbf{n}%
_{k-1}),\supseteq )$ is isomophic to $FD(k)$, the free
distributive lattice with $k$ generators. It is a famous problem,
raised by Dedekind, to give an explicit and workable formula for
$FD(k)$.  The largest exact value known is  $FD(8)$ \cite {wiedemann}. An asymptotic formula was given by Korshunov  in 1981\cite{korshunov}.

   For an example, on the two-letter alphabet $A=\left\{ a,b\right\} $,
the words $u_{0}$ $:=$ $aa$ and $u_{1}:=$ $bb$ give the lattice (ordered by reverse of inclusion) and graph
represented on Figure 1.

\unitlength=1cm  
\begin{picture}(16,7)
\put(2,6.5){\circle*{.15}}
\put(2,5.5){\circle*{.15}}
\put(1,4.5){\circle*{.15}}
\put(3,4.5){\circle*{.15}}
\put(2,3.5){\circle*{.15}}
\put(2,2.5){\circle*{.15}}
\put(2,6.5){\line(0,-1){1}}
\put(2,5.5){\line(1,-1){1}}
\put(2,5.5){\line(-1,-1){1}}
\put(1,4.5){\line(1,-1){1}}
\put(3,4.5){\line(-1,-1){1}}
\put(2,2.5){\line(0,1){1}}
\put(2.2,2.5){$A^{*}$}
\put(3.2,4.5){$\uparrow \{ b,aa\}$}
\put(2.2,5.5){$\uparrow \{ab,ba,aa,bb\}$}
\put(2.2,6.5){$\uparrow \{ bb,aa\}$}
\put(-.7,4.5){$\uparrow \{a,bb\}$}
\put(2.4,3.5){$\uparrow \{a,b\}$}
\put(0,1.5){The lattice structure of ${\mathcal S}_F$}
\put(9,6.5){\circle*{.15}}
\put(9,5.5){\circle*{.15}}
\put(8,4.5){\circle*{.15}}
\put(10,4.5){\circle*{.15}}
\put(9,3.5){\circle*{.15}}
\put(9,2.5){\circle*{.15}}
\put(9,6.5){\line(0,-1){1}}
\put(9,5.5){\line(1,-1){1}}
\put(9,5.5){\line(0,-1){2}}
\put(9,5.5){\line(-1,-1){1}}
\put(8,4.5){\line(1,-1){1}}
\put(8,4.5){\line(1,0){2}}
\put(10,4.5){\line(-1,-1){1}}
\put(9,2.5){\line(0,1){1}}
\bezier{200}(9,6.5)(10.5,5.5)(10,4.5)
\put(10,5.5){\vector(1,-2){.1}}
\bezier{200}(9,6.5)(7.5,5.5)(8,4.5)
\put(8,5.5){\vector(2,3){.1}}
\bezier{200}(10,4.5)(10.5,3.5)(9,2.5)
\put(10,3.5){\vector(-1,-2){.1}}
\bezier{200}(9,2.5)(7.5,3.5)(8,4.5)
\put(8,3.5){\vector(-1,3){.1}}
\put(7,1.5){The graphic structure of ${\mathcal S}_F$}
\put(8.5,2){\vector(1,1){.3}}
\put(9.1,6.5){\vector(-1,1){.3}}
\put(4.5,.5){Figure 1}
\end{picture}    

%\begin{figure}[htbp]
%\includegraphics[width=14pc,height=14pc]{fig1left} 
%
%\includegraphics[width=14pc,height=14pc]{fig1right} 

%The lattice structure of $\mathcal{S_F}$  & The graphic structure
%of $\mathcal{S_F}$
%
%
%\end{figure}
%\centerline{Figure 1}
%\vspace{0.5cm}

 Structural properties of transition systems
rely upon algebraic properties of languages and conversely. In
fact, transition systems can be viewed as geometric objects
interpretating these algebraic properties. An illustration of this claim is given by the following result (see Corollary 4.9 \cite{kabil-pouzet}). 

\begin{theorem} Let $F$ be a nonempty final segment of $A^{\ast }$. If
$F$ is the  concatenation  of final segments  $F_{1}, \dots , F_{n}$  then the automaton $\mathcal A_F:= (M_{F}, \{x\}, \{y\})$ associated to the injective envelope 
 $\mathcal S_{F}$  is the concatenation  
$\mathcal A_{F_{1}}\cdot \cdot\mathcal A_{F_{n}}$ of automata  $\mathcal A_{F_i}:= (M_{F_i}, \{x_i\}, \{y_i\})$ associated to the injective envelope $S_{F_i}$, this concatenation being obtained by identifying  each $y_i$ with $x_{i+1}$. \end{theorem}

The fact that an 
automaton decomposes into such a concatenation  can be viewed directly by
looking at states which disconnect the underlying graph. From this follows  the  uniqueness of such a decomposition. This uniqueness amounts to the fact that the monoid  $\mathbf F(A^{*})\setminus \{\emptyset\}$ is free. A purely algebraic proof of this result is given in \cite{kabil-pouzet-rosenberg}. 

We discuss then the relationship between the minimal deterministic automaton accepting  a final segment $F$, say $Min_F$,   and the automaton $\mathcal A_{F}$ associated to the injective envelope $\mathcal S_{F}$. This minimal automaton is part of $\mathcal A_{F}$, but not in an isometric way ($Min_F$ being deterministic cannot be reflexive,  in general it is not involutive). Among involutive and reflexive transition systems accepting a given final segment $F$, we consider those with a minimum number of states and among those, the ones with a maximal number of transitions, that we call Minmax automata. Exemples given by Mike Main and communicated by Maurice Nivat \cite{nivat} show that contrarily to the case of deterministic automata, these automata are not unique.

We introduce Ferrers languages.  A language $L$ over $A^*$ is \emph{Ferrers} if 
\begin{equation}
  xx^{\prime }\in L\;  \mbox{and}\; yy^{\prime }\in L\;  \text{imply}\;  xy^{\prime }\in L\; 
\text{or}\;  yx^{\prime }\in L\;  \text{for all}\; x,x^{\prime },y,y^{\prime }\in
A^{\ast }.   
\end{equation} 

The class of Ferrers languages is closed under complement but not under concatenation. Still, if $F_1, \dots F_{n}$ are Ferrers and each $F_i$ is   a final segment of $A^{*}$ then the concatenation $F_1\cdots F_n$ is Ferrers (Corollary \ref{productFerrers}).

We prove that a final  segment $F$ of $A^{*}$  is Ferrers  if and only if the injective envelope $\mathcal S_{F}$ is totally orderable, that is there is a linear order $\preceq$ on  $\mathcal S_{F}$ such that $d(x,z)\preceq d(x,y)$ and $d(z,y)\preceq d(x,y)$ for all $x\preceq z\preceq y$. (Theorem \ref{prop:ferrerslanguage}). 

Over a finite alphabet $A^*$, Boolean combinations of final segments of $A^{*}$ are called  \emph{piecewise testable languages}. They have been characterized by Simon \cite{simon} by the fact that their syntactical monoid is $\mathcal J$-trivial. The Boolean algebra of piecewise testable languages is included into the Boolean algebra generated by rational Ferrers languages. Indeed,  over a finite alphabet, every final segment is a finite union of rational Ferrers languages. But, on an alphabet with at least two letters,  there are  rational Ferrers languages which are not piecewise testable (e.g.,   $L:= A^*b$ on $A:= \{a,b\}$). We do not know if  they are dot-depth one.

%For the converse, it suffices to show that every rational Ferrers language is piecewise testable. According to Straubing and Th\'erien (Partially ordered finite monoids and a Theorem of I.Simon, J.of Algebra 119, 393-399(1988)), a finite monoid is $\mathcal J$-trivial if and only if it is a quotient of a finite ordered monoid. Our question reduces then to show that the syntactical  monoid of a rational Ferrers language $F$ is a quotient of a partially ordered monoid. The fact that $F$ is rational and Ferrers amounts to say that the minimal state automaton, made of residuals of the form $u^{-1} F$, is finite and totally ordered by inclusion. Does this implies the answer we are looking for?

%The study of metric spaces over a Heyting algebra was inspired by  the work of Quilliot 1983.  It started in 1983 with the theses  of Jawhari and  Misane under the supervison of the second author; it was followed by several papers \cite{pouzet} 1984, \cite{JaMiPo} 1986, \cite{pouzet-rosenberg} 1994, \cite {saidane} 1992, \cite {kabil-pouzet} 1998,  \cite {kabil-pouzet-rosenberg} 2018. A study of retraction, coretraction and injective objects among transition systems was also developped by Hudry \cite{hudry1, hudry2}. The existence of injective envelope was proved in \cite{JaMiPo}; \cite{kabil-pouzet}, \cite{kabil} and also \cite {kabil-pouzet-rosenberg},  contain  a large account  of its properties. 

This paper is organized as follows. Properties of
metric spaces over a Heyting algebra and their injective envelopes
are summarized in section 2.
 In section 3, we introduce  the Heyting algebra $\mathbf F(A^{\ast })$. In section 4 we consider transition systems as metric spaces.  In section  5 we describe the injective envelope of a 
two-element metric spaces over $\mathbf F(A^{\ast })$; we prove Theorem \ref{thm:main1} and  \ref{injenv} and conclude the section by a counterexample about Minmax automata due to M.Main. Ferrers languages are introduced in section 6.

The  results  developped here have been presented at the International Conference on Discrete Mathematics and Computer Science (DIMACOS'11) organized by A.~Boussa\"{\i}ri, M.~Kabil, and A.~Taik  in Mohammedia (Morocco) May, 5-8, 2011. They were never published; a part of it was included  into the Th\`ese d'\'Etat defended by the first author \cite{kabil}.

\section{Metric spaces over a Heyting algebra}\label{section metric}
\subsection{Basic facts}
  The following is extracted from \cite{kabil-pouzet-rosenberg} (for more details, see  \cite {kabil-pouzet}). Let $\mathcal H $ be
a Heyting algebra  and let $E$ be a set. A \textit{$\mathcal H $-distance} on $E$ is a map $%
d:E^{2}\longrightarrow \mathcal H $ satisfying the following properties for all $%
x,y,z\in E$:
\begin{enumerate}[(1)]
\item  $\ d(x,$ $y)=1\Longleftrightarrow x=y$, 

\item $\ d(x,y)\leq d\left( x,z\right) \cdot d\left( z,y\right) $, 

\item $d(x,$ $y)=\overline{d\left( y,x\right) }.$
\end{enumerate}

The pair $\left( E,d\right) $ is called a $\mathcal H $-\emph{metric space}. If there is no
danger of confusion we will denote it $E.$ A $\mathcal H $-distance can be defined on $%
\mathcal H $. This fact relies on the classical notion of \textit{residuation}. Let $v\in \mathcal H $. Given $\beta \in \mathcal H $, each of the sets
 $\{r \in \mathcal H : v \leq r \cdot \beta\}$ and
$\{r \in \mathcal H : v \leq  \beta \cdot r \}$ has a least element, that we 
denote
respectively $\lceil v\cdot \beta^{-1} \rceil$ and $\lceil\beta^{-1} \cdot v  \rceil$
(note that  $\overline {\lceil\beta^{-1} \cdot v \rceil} =
\lceil \bar v\cdot  (\bar\beta)^{-1} \rceil$). It follows that for all
$p, q \in \mathcal H $, the set
$$D(p,q):=\{r \in \mathcal H  : p\leq q \cdot \bar r\;\;{\rm and}\; \; q\leq 
p\cdot r\}$$
has a least
element, namely $\lceil \bar p\cdot (\bar q)^{-1} \rceil \vee
\lceil p^{-1} \cdot q \rceil$, that we denote $d_\mathcal H (p,q)$. As shown in \cite{JaMiPo}, the
map $(p,q) \longrightarrow d_{\mathcal H }(p,q)$
is a $ \mathcal H -$distance.

Let $(E,d)$ and $\left( E^{\prime },d^{\prime }\right) $ be two
$\mathcal H -$metric spaces. Recall that a map $f:E\longrightarrow E^{\prime }$ is a
\textit{contraction}  (or a {non-expansive map}) from $(E$, $d)$ to $\left( E^{\prime },d^{\prime
}\right) $ provided that $d^{\prime }(f(x),(f(y))\leq d(x,y)$ holds for all $%
x,y,\in E$.  The map $f$ is an \textit{isometry} if $d^{\prime
}(f(x),(f(y))=d(x,y)$ for all $x,y,\in E$. We say that $E$ and
$E^{\prime }$ are \textit{isomorphic}, a fact that we denote
$E\cong E^{\prime }$, if there is a surjective isometry from $E$
onto $E'$.

Let $\left( (E_{i}, d_{i})\right) _{i\in I}$ be a family of $\mathcal H $-metric spaces. The \emph{direct product}  $\underset{i\in I}{%
\prod }\left( E_{i}, d_{i}\right) $, is the metric space $(E,d) $ where $E$ is the cartesian product  $%
\underset{i\in I}{\prod }E_{i}$ and $d$ is the  ''sup'' (or
$\ell ^{\infty }$) distance  defined
by $d\left(
\left( x_{i}\right) _{i\in I},\left( y_{i}\right) _{i\in I}\right) =%
\underset{i\in I}{\bigvee }d_{i}(x_{i},$ $y_{i})$.

For a $\mathcal H $-metric space $E$, $x$ $\in E$ and $r\in \mathcal H $, we define the {\it ball} $B_{E}(x,r)$ as the set \{$%
y\in E:d\left( x,y\right) \leq r\}$. 
We say that  $E$ is \emph{convex} if the intersection of two balls $B_{E}(x_{1}$, $r_{1})$ and $B_{E}(x_{2}$, $r_{2})$ is non-empty provided that 
$d(x_{1}$, $x_{2})\leq $ $r_{1}\cdot\overline{r_{2}}$. We say that 
$E$ is  \textit{hyperconvex} if the intersection
of every family of balls ($B_{E}(x_{i}$, $r_{i})$)$_{i\in I}$ is non-empty
whenever $d(x_{i}$, $x_{j})\leq $ $r_{i}\cdot\overline{r_{j}}$ for all $i,j\in I$.
For an example, 
\emph{$(\mathcal H ,d_{\mathcal H })$ is a hyperconvex  $\mathcal H $-metric space
and every $\mathcal H $-metric space embeds isometrically into a
power of $(\mathcal H ,d_{\mathcal H })$} \cite{JaMiPo}.  This is due to the fact that for
every $\mathcal H $-metric space $(E,d)$ and for all $x,y\in E$ the
following equality holds:

$$d\left( x,y\right)  = \underset{%
z\in E}{\bigvee } d_{\mathcal H } (d(z,x),d(z,y)).$$

The space $E$ is a \textit{retract} of $E^{\prime }$, in symbols $
E\vartriangleleft E^{\prime }$, if there are two contractions $f:
E\longrightarrow E^{\prime }$ and $g:E^{\prime
}\longrightarrow E$ such that $g\circ f=id_{E}$ (where $id_{E}$ is
the identity map on $E$). In this case,
$f$ is a  \textit{coretraction} and $g$ a \textit{%
retraction}. If $E$ is a subspace of $E^{\prime }$, then clearly $E$ is a retract of
$E^{\prime }$ if there is a contraction from $E^{\prime }$ to $E$ such $%
g(x)=x $ for all $x$ $\in E.$ We can easily see that every coretraction is
an isometry.  A metric space is an \textit{absolute retract} if it is a
retract of every isometric extension. The space $E$ is said to be \textit{%
injective} if for all $\mathcal H $-metric space $E^{\prime }$ and $E'', $ each
contraction $f:E^{\prime }\longrightarrow E$ and every isometry $g:E^{\prime
}\longrightarrow E''$ there is a contraction $h:E''\longrightarrow E$ such that $%
h\circ g=f$.  We recall that \emph{for a metric space over a Heyting algebra $\mathcal H $,
the notions of absolute retract, injective, hyperconvex and retract of a
power of $(\mathcal H ,d_{\mathcal H })$ coincide} \cite{JaMiPo}.

\subsection{Injective envelope}A contraction $f: E\longrightarrow E'$ is \textit{essential} it for
every contraction $g: E'\longrightarrow E''$,  the map $g\circ f$ is
an isometry if and only if $g$ is isometry (note that, in
particular, $f$ is an
isometry). An essential contraction $f$ from $E$ into an injective $\mathcal H $-metric space $E'$ is called an {\it injective envelope} of $E$. We will
rather say that $E'$ is  \textit{an injective envelope}  of  $E$. We
can view an injective envelope of a metric space $E$ as a minimal
injective $\mathcal H $-metric space containing (isometrically) $E$.  Two injective envelopes of $E$ are   isomorphic via an isomorphism which is the identity over $E$. This allows to talk about "the" injective envelope of $E$; we will denote it by $\mathcal N(E)$. A particular injective envelope of $E$ will be called a \emph{representation} of $\mathcal N(E)$. The construction of  
injective envelope is based upon the notion of\textit{
minimal metric form}. A \emph {weak metric form} is every map $
f:E\longrightarrow \mathcal H $ satisfying $d(x,y)\leq f(x)\cdot \overline {f(y)}$ for all $x,y\in E$.  This is a \emph{metric form} if in addition $f(x) \leq d(x,y) \cdot f(y)$ for all $
x,y\in E.$ Equivalently, $f$ is a metric form if and only if  $d_{\mathcal H }(d(x,y),f(x))\leq f(y)$ for all $
x,y\in E$. A (weak) metric form is  \textit{minimal} if there is no other  (weak) metric form $g$
satisfying $g\leq f$ (that is $g(x)\leq f(x)$ for all $x\in E$). Since every weak metric form majorizes a metric form,  the two notions of minimality coincide.  As shown in
\cite{JaMiPo} \emph{every $\mathcal H $-metric space
has an injective envelope;   the space of
minimal metric forms is a representation of it, (cf. Theorem 2.2 of \cite{kabil-pouzet})}.

We give below a new characterization of the injective envelope. 

\begin{theorem}\label{thm:main2}
Let $(E, d)$ be a metric space over $\mathcal H$ and  $X\subseteq E$.  
Then $(E, d)$ is isomorphic to the injective envelope of $(X, d_{\restriction X})$  iff for every metric space   $(E^{\prime}, d')$, $X'\subseteq E'$, every isometry $f: (X, d_{\restriction X})$ onto $(X', d'_{\restriction X'})$:

\begin{enumerate}[{(i)}]
\item Every non-expansive map    $\overline f: (E, d)\rightarrow (E',d')$, if any,  which extends $f$ is  an isometric embedding of $(E,d)$ into $(E',d')$; 

\item The isometry $f^{-1}$ extends to a non-expansive map from  $(E', d')$ into $(E, d)$. \end{enumerate}\end{theorem}

\begin{proof}
Suppose that $(E, d)$ is isomorphic to the injective envelope of $(X, d_{\restriction X})$. Since $(E, d)$ is injective, the isometry $f^{-1}$ extends to a non-expansive map from  $(E', d')$ into $(E,d)$.  This proves that $(ii)$ holds.  The proof that $(i)$ holds relies on the properties of metric forms. To prove that $f$ is an isometry embedding amounts to  prove the equality: 
\begin{equation}\label{eq:isom1}
d'(f(x), f(y))= d(x,y)
 \end{equation}
  for  all $x,y \in E$. Set  $X'':= f(X)$. For every $x\in E$, let  $h_x: X\rightarrow \mathcal H$  and $g_{x}: X\rightarrow \mathcal H$ be the maps defined by setting $h_x(z)=d(z,x)$  and $g_{x}(z)=d'(f(z), f(x))$ for $z\in X$. Since $E$ is the injective envelope of $X$, $h_x$ is a minimal metric form; since $f$ is non-expansive and induces  an isometry from $X$ onto $X''$, the map $g_{x}$  is a  metric form below $h_x$, hence $h_x=g_x$. It follows that 
  \begin{equation}\label{eq:isom2}
  d(z, x)= d(f(z), f(x))
  \end{equation}
  for every $z\in X$. 
  
  Let $x,y\in E$. By construction of the injective envelope, its elements identify to minimal metric forms over $X$, hence $x$ and $y$ identify respectively to $h_x$ and $h_y$ and \begin{equation}\label{eq:isom3}
 d (x,y)= d_{\mathcal H}(h_x,h_y):= \Sup\{d_{ \mathcal H}( d(z, x), d(z,y)): z\in X\}. 
 \end{equation}
  
  Set $\alpha:= d'(f(x), f(y))$. Let $z\in E$. By definition,  $d_{\mathcal H}( d'(f(z), f(x)), d'(f(z), f(y))) = \Min D( d'(f(z), f(x)), d'(f(z), f(y)))$ where 
  $D(d'(f(z), f(x)), d'(f(z), f(y)))= \{r\in \mathcal H: d'(f(z), f(x))\leq d'(f(z), f(y))\cdot \overline r \; \text{and}\; d'(f(z), f(y))\leq d'(f(z), f(x))\cdot r\}$. From the triangular inequality, we have 
  $d'(f(z), f(y))\leq d'(f(z), f(x))\cdot \alpha$ and $d'(f(z), f(x))\leq d'(f(z), f(y))\cdot \overline \alpha$. Hence, $\alpha \in D( d'(f(z), f(x)), d'(f(z), f(y)))$.
   \noindent The inequality $$d_{\mathcal H}( d'(f(z), f(x)), d'(f(z), f(y))) \leq \alpha$$ follows.
 Consequently,    
  $$\Sup \{d_{\mathcal H}( d'(f(z), f(x)), d'(f(z), f(y))): z\in E\} \leq \alpha.$$  
 With Equalities   (\ref{eq:isom2}) and (\ref{eq:isom3})we get:
 $d(x,y)= \Sup \{d_{\mathcal H}( d(z, x), d(z, y)): z\in E\} \leq \alpha$.
 Since $f$ is non expansive, we have $\alpha\leq d(x,y)$. Thus $d(x,y)= \alpha$ that is Equality (\ref{eq:isom1}) holds. 
  
 Conversely, suppose that $(i)$ and $(ii)$ hold. Let   $(E', d')$  be the injective envelope of $(X, d_{\restriction X})$ and $f$ be the identity map from $(X, d_{\restriction X})$ onto itself. Applying $(ii)$, the map $f^{-1}$ extends to a non-expansive map $g$ from $(E', d')$ into $(E,d)$. Since $(E',d')$ is  injective, the map $f$ extends to a non-expansive map $\overline f$ from $(E,d)$ into $(E',d')$.   The map $f\circ g$ is non-expansive and is the  identity on $X$. Since $(E',d')$ is the injective envelope of  $(X, d_{\restriction X})$, $f\circ g$ is the identity on $E'$ (note that elements of $E'$ identify to minimal metric forms over $X$), hence $g$ is injective and $f$ is surjective. Now by $(i)$, $f$ is an isometry on its image. Hence $f$ is an isometry of $(E,d)$ onto $(E',d')$. Thus $(E,d)$ is the injectyive envelope of $(X, d_{\restriction X})$ as claimed. 
\end{proof}

Up to Theorem \ref{thm:indec}, we include the few facts we need about injective envelopes
of two-element metric spaces ( see \cite{kabil-pouzet} for proofs).

Let $\mathcal H $ be a Heyting algebra and $v\in \mathcal H $. Let $E:=\{x,y\}$ be a two-element $\mathcal H $-metric space such that $d(x,y)=v$. We denote by $\mathcal{N}_v$  the  injective envelope of $E$. We give three representations of it.  
For the fist one, we consider the set of minimal metric forms over $E$.  That is,  in this case, the set of minimal pairs $h:= (h_x,h_y)\in \mathcal H^2$ such that $h_x\cdot \overline h_y\geq v$,  the set $\mathcal H^2$ being equipped with the product ordering. Each element $z\in N_v$ identifies to the pair $(d_{\mathcal N_v}(x,z), d_{\mathcal N_v}(y,z))$; in particular, $x$ and $y$ identify to $(1, \overline v)$ and to $(v, 1)$ respectively. We equip $\mathcal H ^{2}$
with the supremum distance:
$$d_{\mathcal H ^{2}}\left( (u_{1},u_{2}), (u'_{1},u'_{2})\right)
:=d_{\mathcal H }(u_{1}, u'_{1})\vee d_{\mathcal H }(u_{2}, u'_{2}).$$

With the induced distance, $\mathcal N_v$ becomes a metric space. If $(h_x,h_y)$ and $(h'_x,h'_y)$ are two elements of $\mathcal N_v$, their distance is $d_{\mathcal H }(h_x, h'_{x})\vee d_{\mathcal H }(h_{y}, h'_{y}).$ In fact, 

\begin{equation} \label{eq:duality}d_{\mathcal H }(h_x, h'_{x})= d_{\mathcal H }(h_{y}, h'_{y}).
\end{equation}

The proof is easy: We  prove that if  $h_x\leq h'_x\cdot \overline r$ for some $r\in \mathcal H$, then $h'_y\leq h_y\cdot r$. This and the corresponding inequality with $y$ replacing $x$ will leads to  (\ref{eq:duality}).  Suppose that $h_x\leq h'_x\cdot \overline r$. We have $v\leq h_x\cdot \overline h_y\leq h'_x\cdot \overline r\cdot \overline h_y$ and thus $r\leq h'_x\cdot \overline r\cdot \overline h_y$. Since $v\leq h'_x\cdot \overline h'_y$, we have $v\leq h'_x \cdot (\overline h'_y\wedge (\overline r\cdot \overline h_y))$ by distributivity. Since $(h'_x,h'_y)$ is a minimal metric form above $v$, we have $\overline h'_y=\overline h'_y\wedge (\overline r\cdot \overline h_y)$, that is $h'_y\leq h_y\cdot r$.

Due to the fact that in a minimal metric form  $(h_x, h_y)$ each component determines the other, we may prefer an other presentation of $\mathcal N_v$ as a subset of $\mathcal H$.  
Set $\mathcal{S}_{v}:=  \left\{ \lceil v\cdot  \beta^{-1} \rceil 
:\beta \in \mathcal H \right\}$;  equipped with the ordering induced by
the ordering over $\mathcal H $ this is  a complete lattice. A pair  $(h_x, h_y)$ belongs to $\mathcal{N}_{v}$ if and only if   $h_x= \lceil v\cdot \overline {h_y}^{-1} \rceil$ and $\overline h_y =\lceil h_x^{-1}\cdot v \rceil$. This yields a correspondence between   $ \mathcal {N}_{v}$ and  $\mathcal{S}_{v}$.

Now, in several instances, e.g. in the case of the sum of two metric spaces (see subsection \ref{subsection:sum}),  it is preferable to consider the set  $\mathcal{C}_{v}$ of of all pairs
$(u_{1},u_{2})\in \mathcal H ^{2}$ such that $v\leq u_{1}\cdot u_{2} $. Once equipped with the ordering induced by the product ordering on $\mathcal H ^{2}$, we can consider the set   $%
\mathcal{N'}_{v}$  of its minimal elements. Each minimal element   $(u_1, u_2)$ yields the  minimal metric form $(u_1, \overline {u}_2)$ (and conversely). The distance over $\mathcal H^{2}$ is different from the previous case.  We have to equip $\mathcal H ^{2}$
with the product of the distance $d_{\mathcal H}$ with the distance $d'_{\mathcal H}$ defined on $\mathcal H$ by $d'_{\mathcal H} (u_2, u'_2):={d_{\mathcal H}(\overline u_2, \overline u'_2)}$.

\begin{lemma} (Lemma 2.3,  Proposition 2.7 of \cite {kabil-pouzet})
The space $\mathcal{N}_{v}$ equipped with the supremum
distance and the set $\mathcal{S}_{v}$  equipped with
the distance induced by the distance over $\mathcal H $   are injective envelopes of   the two-element metric spaces $\{(1, v), (v,1)\}$ and  $\left\{
1,v\right\}$ respectively. These spaces are isometric to the injective envelope of $E:= \{x,y\}$. 
\end{lemma}

%The injective envelope being defined up to an isomorphism, there is 
%no reason 
%to make difference between ${\mathcal N}_v$ and ${\mathcal S}_v$, despite the 
%fact that their definitions in terms of sets are different. For 
%commodity 
%reasons, we use one rather than the other depending the circumstances.
%

\begin{theorem} (Theorem 2.9 \cite{kabil-pouzet})\label{thm:injective-proper}
 Let $E$ be a $\mathcal H $-metric space. If $E$ is the injective envelope
of a two-element set then $E$ contains no proper isometric
subspace.
\end{theorem}

\begin{proposition} (Corollary 3.3 \cite{kabil-pouzet})\label{prop:finiteness}
The following properties are equivalent:
\begin{enumerate}[(i)]

\item  The injective envelope of any finite metric space is finite;

\item The injective envelope of any two-element metric space   is finite.

\end{enumerate}
\end{proposition}

 A metric space is \textit{linearly orderable} if there is a
linear ordering $\preceq $ on $E$ such that $d(x,z)\preceq d(x,y)$ and $d(z,y)\preceq d(x,y)$ for all $%
x,y,z\in E$ with $x\preceq z\preceq y.$

\begin{proposition}  (Fact 5 \cite{kabil-pouzet})\label{prop:linearly}
 Let $u\in \mathcal H $. The space  $\mathcal{S}_{u}$ is linearly orderable if and only if  the ordering is induced by the order on $\mathcal H $ or by its  reverse. 
\end{proposition}

We say that a metric space is \textit{finitely indecomposable} if for every
finite family $(E_{i})_{i\in I}$ of metric spaces, $\underset{i\in I}{%
E\vartriangleleft \prod }E_{i}$ implies $E\vartriangleleft E_{i}$ for some $%
i\in I.$ This notion was used by E. Corominas for posets \cite{corominas}.

\begin{theorem}(Theorem 3.8 \cite{kabil-pouzet})\label{thm:indec}
Let $E$ be a finite absolute retract. The following properties are
equivalent:
\begin{enumerate}[(i)]
\item $E$ is finitely indecomposable;
\item $E$ is
the injective envelope of a two-element metric space $\{x,y\}$
such that the distance $ d(x,y)$ is join-irreducible in $\mathcal H $.
\end{enumerate}
\end{theorem}

\subsection{Sum of metric spaces}\label{subsection:sum} 
%We give an algebraic condition which allows to describe $\mathcal{N}%
%_{v_{1}\cdotv_{2}}$ in terms of $\mathcal{N}_{v_{1}}$ and $\mathcal{N}_{v_{1}}$.
%

Let $(E_{1},d_{1})$ and $(E_{2},d_{2})$ be two disjoint $\mathcal H $-metric spaces
and let $x_{1}\in E_{1}, x_{2}\in E_{2}$. If we endow the set $\{x_{1}, x_{2}\}$ with
a $\mathcal H $-distance $d^{\prime }$, then we can define a $\mathcal H $-distance $d$ on $%
E:=E_{1}\cup E_{2}$ as follows:

$\bullet$ If $x,y\in E_{i}$  with $i\in \left\{
1, 2\right\}$  then $d$ ($x,y$) = $d_{i}$ ($x,y$); 

$\bullet$ If $x\in E_{i}$, $y\in E_{j}$ with $i,j\in \left\{ 1, 2\right\} $ and $i\neq
j$,  then 
$d(x,y) = d_{i}(x,x_{i}) \cdot  d^{\prime }(x_{i},x_{j}) \cdot d_{j}(x_{j},y).$ In particular, we can identify $x_{1}$ and $x_{2}$ which amounts to set $%
d^{\prime }(x_{1},x_{2})=0$ in the above formula.

If $E_{1}$ and $E_{2}$ are not disjoint, we replace it by two disjoint
copies $E_{1}^{\prime }$, $E_{2}^{\prime }($eg $E_{i}^{\prime
}: =E_{i}\times \left\{ i\right\} )$. Identifying the corresponding elements $x_{1}^{\prime },x_{2}^{\prime
} $, we obtain a $\mathcal H $-metric space that we denote $%
(E_{1},d_{1})\cdot(E_{2},d_{2})$.  Alternatively, we may suppose that $E_{1}$ and $E_{2}$ have only one element in common, say $z_{1,2}$, and we define the distance $d$ on $E_{1}\cup E_{2}$ by setting $d(x,y):= d_i(x,z_{1, 2})\cdot d_j(z_{1,2},y)$ if $x\in E_i$, $y\in E_j$, $i\not =j$, and $d(x,y):=d_i(x,y)$ if $x,y\in E_i$. 

We consider now objects consisting of a $\mathcal H $-metric space and two
distinguished elements. Given two such objects, say $\underline E_{1}:=\left(
\left( E_{1},d_{1}\right), x_{1},y_{1}\right) $ and $\underline E_{2}:=$
$\left( \left( E_{2},d_{2}\right), x_{2},y_{2}\right), $ set
$\underline E_{1}\cdot \underline E_{2}:=\left( \left( E,d\right),x,y\right) $ where ($E,d$)
is the space obtained by taking
disjoint copies $(E_{1}^{\prime},d_{1}^{\prime})$ and $%
(E_{2}^{\prime},d_{2}^{\prime})$ of $(E_{1},d_{1})$ and $%
(E_{2},d_{2}),$respectively, and by identifying the corresponding elements $%
y_{1}^{\prime }$and $x_{2}^{\prime}$ and setting $x:=x_{1}^{\prime
} $ and $y:=y_{2}^{\prime}.$

\begin{definition}
Let $\mathcal H $ be Heyting. A pair $(v_1, v_2)\in \mathcal H ^2$ is \emph{summable} if  $\mathcal{S}%
_{v_{1}\cdot v_{2}}$  is isomorphic to  $\mathcal{S}_{v_{1}}\cdot \mathcal{S}_{v_{1}}$.
\end{definition}

\begin{definition}\label{decomposition property} Let $\mathcal H '$ be an initial segment of $\mathcal H $ which is also a submonoid. We say that  $\mathcal H '$ has the \textit{decomposition property}
if every pair $(v_1, v_2) \in \mathcal H ' \times \mathcal H '$ is summable.
\end{definition}

Several examples of metric spaces over a Heyting algebra are given in \cite {JaMiPo}.
We briefly examine some of these examples w.r.t. to their injective envelopes and the sum operation. 
Ordinary metric spaces enter in his frame. Add a largest element $+\infty$ to the set  $\R^{+}$ of non-negative reals,  extend the $+$ operation in the natural way, take the identity for the involution.  Then $\R^{+} \cup \{+\infty\}$ becomes a Heyting algebra and the metric spaces over it are just direct sums of ordinary metric spaces. The injective envelope of such a space is the direct sum of the injective envelope of its factors.  The injective envelope of a two-element metric spaces $E:=\{x,y\}$ with $r:= d(x,y)$ is isometric  to the segment $[0, r]$ si $r<+\infty$ and to $E$ if $r= +\infty$. Trivially, $\R^{+}$ has the decomposition property. The sum of two convex (resp. injective)  metric spaces with a common vertex is convex (resp. injective).  The reader will find in
\cite{Dr1} a description of injective envelopes of finite
ordinary metric spaces and interesting combinatorial properties as
well (see also \cite {Dr2}). \\

Now, let $\mathcal H := \{a, b, 0,1\}$ ordered by  $0< a,b<1$ and $a$ incomparable to $b$. The operation is the join, the involution exchange $a$ and $b$. Metric spaces over $\mathcal H $ correspond to ordered sets. As shown by Banaschewski and Bruns, every poset $P$  has an injective envelope, namely its MacNeille completion \cite{banaschewski-bruns}. Hence, if $P$ has   two elements, its injective envelope is  $P$ whenever  these two elements are comparable,  otherwise this is  $P$ augmented of a smallest and a largest element. Convexity property does not hold for $\mathcal H $, hence the sum of two injective with a common element does not need to be injective. 

Next,  suppose that $\mathcal H $ is a complete meet-distributive lattice, the operation is the join and the involution is the identity. For example, if  $\mathcal H :=\R \cup \{+\infty\}$,  metric spaces over $\mathcal H $ are direct sums of ultrametric spaces. Metric spaces over  Boolean algebras have been introduced  by Blumenthal \cite{blumenthal}. Let $B$ be a Boolean algebra, let the operation be the supremum and the involution be the identity.  Although $B$ is not necessary complete, residuation allows to define a distance on $B$ setting $d_B(p,q):= p\Delta q$ where $\Delta$ denotes the symmetric difference. From this follows that the interval $[0, u]:= \{v\in B: 0\leq v\leq u\}$ with the distance induced by $d_B$ is the injective envelope of every pair $\{p, q\}$ of vertices of  $B$ such that $p\wedge q= 0$ and $p\vee q= u$. Hence if $B$ is finite, the number of elements of the injective envelope of a $2$-element metric space is a power of $2$ hence the decomposition property does not hold. In Section \ref{section heyting},  we give an example for which this decomposition property holds, namely  the  algebra $\mathbf F^{\circ} \left( A^{\ast }\right)$(for more examples of generalisations of metric spaces, see \cite {blumenthal1}, \cite{blumenthal-menger}).

\section{The Heyting algebra $\underline{\mathbf F}\left( A^{\ast }\right) $}\label{section heyting}

Let $A$ be a set. Considering $A$ as an \textit{alphabet} whose
members are \textit{letters}, we write a word $\alpha $ with a
mere juxtaposition of its letters as $\alpha
=a_{0}a_{1}...a_{n-1}$ where $a_{i}$
are letters from $A$ for 0 $\leq i\leq i-1.$ The integer $n$ is the \textit{%
length} of the word $\alpha $ and we denote it $\left| \alpha \right| $.
Hence we identify letters with words of length 1. We denote by $\Box $ the
empty word, which is the unique  word of length zero. 
The \emph{concatenation} of two word $\alpha:= a_{0}\cdots a_{n-1}$ and $\beta:=b_{0}\cdots b_{m-1}$ is the word $\alpha \beta:=a_{0}\cdots a_{n-1}b_{0}\cdots b_{m-1}$. We denote by $A^{\ast }$ the set of all words on the alphabet $A$. Once equipped with the
concatenation of words, $A^{\ast }$ is a monoid, whose neutral element is the empty word, in fact $A^{\ast}$ is the \textit{free
monoid} on $A$.   A \emph{language} is any subset $X$ of $A^{\ast}$. We denote by $\powerset (A^{\ast})$ the set  of languages. We will use capital letters for languages. If $X, Y \in \powerset (A^{\ast})$ we may set  $XY:= \{\alpha\beta: \alpha\in X, \beta\in Y\}$ (and use  $Xy$ and $xY$instead of $X\{y\}$ and $\{x\}Y$).  With this  operation, which extends the concatenation operation on $A^{\ast}$, the set $\powerset (A^{\ast})$ is a monoid (the set $\{ \Box \}$ is the neutral element). Ordered by inclusion, this is a (join) lattice ordered monoid. Indeed, concatenation distributes over arbitrary union, namely:

\begin{center}
$( \underset{i\in I}{\bigcup }X_{i})Y=\underset{i\in I}{ \bigcup }X_{i} Y.$
\end{center}

This monoid is residuated. Let $X, Y, F \in \powerset(A^{\ast})$. As it is customary, we set $X^{-1}F:= \{y\in A^{*}: Xy\subseteq F\}$ and $FY^{-1}:= \{x\in A^{*}: xY\subseteq F\}$. We recall that 
\begin{equation}\label{eq:residuation}
X(X^{-1}F)\subseteq F\;  \text {and}\;  (FY^{-1})Y\subseteq F. 
\end{equation}

Set $\rho_F:= \{(x,y)\in A^*\times A^*\}$ and $R:=(A^*, \rho_F, A^*)$, then $R(x)= x^{-1}F$ and $R^{-1}(y)=Fy^{-1}$. Thus,  
 $\gal(R)$, the  Galois lattice of $R$,  is the set $\{FY^{-1}: Y\subseteq A^{\ast}\}$  ordered by inclusion. This is a complete lattice.  The meet is the intersection, the largest element is $A^{\ast}$.  

In the sequel, we study the metric structure of $\gal (R)$ when  $F$ is a final   segment  
of the monoid $A^{\ast}$, this monoid being  equipped with the Higman ordering.

We suppose from now that the alphabet $A$ is
ordered and equipped with an involution $^{-}$ preserving the order. The
involution extends to $A^{\ast }$: we set for every $\alpha:
=a_{0}\cdots a_{n-1}, \overline{\alpha }:=\overline{a_{n-1}}\cdots \overline{a_{0}}$. Note that $\overline{\alpha  \beta }=\overline{\beta
} \overline{\alpha }$ for all $\alpha,\beta \in A^{\ast }$. We
order  $A^{\ast }$ with the Higman ordering: if $\alpha $ and
$\beta $ are two elements in $A^{\ast }$ such $\alpha: =a_{0}\cdots 
a_{n-1}$ and $\beta: =b_{0}\cdots  b_{m-1}$ then $\alpha \leq \beta$
if there is an injective and increasing map $h$ from $\left\{
0,...,n-1\right\} $ to $\left\{ 0,...,m-1\right\}$ such that for each $i$,  $
0\leq i\leq n-1$, we have $a_{i}\leq b_{h\left( i\right) }$. Then 
$A^{\ast }$ becomes an  ordered monoid with respect to the concatenation
of words.  Let $\mathbf F\left( A^{\ast }\right) $ be the collection of
final segments of $\ A^{\ast }$ (that is $X\in \mathbf F\left( A^{\ast
}\right) $ if $X\subseteq A^{\ast } $ and $\alpha \leq
\beta,\alpha \in X$ implies $\beta \in X)$. The set $\mathbf F\left( A^{\ast }\right)$ is stable w.r.t. the concatenation of languages:  if $ X,Y\in \mathbf F\left(
A^{\ast }\right)$,  then $XY\in \mathbf F(A^{\ast})$.  Clearly,  the neutral element
 is $A^{\ast }$.  The set $\mathbf F\left( A^{\ast }\right) $ ordered by
inclusion is a complete lattice (the join is the union, the meet
is the intersection).  Concatenation distributes over union. 
Order $\mathbf F\left( A^{\ast }\right) $ by reverse of the
inclusion, denote  $X\leq Y$ instead of $X\supseteq Y$, 
extend the involution $^{-}$ to $\mathbf F( A^{\ast })$, set
$\overline{X}=\{\overline{\alpha }: \alpha \in X\}$,  denote by $X\cdot Y$  the concatenation $XY$ and set $1:= A^{\ast}$ then:

\begin{theorem}\label{fact:heyting}  The set $\underline{\mathbf F}\left( A^{\ast
}\right):=\left( \mathbf F\left( A^{\ast }\right),\leq,\cdot,1,^{-}\right) $ is a
Heyting algebra.
\end{theorem}

We may then define metric spaces over $\underline{\mathbf F}\left( A^{\ast
}\right)$ and study injective objects and particularly injective envelopes.

According to Corollary 4.9 p.177 of \cite{kabil-pouzet} we have:

 \begin{theorem}\label{thm:sum} 
Let $F$, $F_{1}$, $F_{2}$ be final segment of \ $A^{\ast }$ such that $ 
F=F_{1}\cdot F_{2}$.  If $F\neq \emptyset $ then

$$ F=F_{1}\cdot F_{2}\ \Longleftrightarrow \mathcal{ S}_{F}\cong \mathcal{S}_{F_{1}}\cdot
\mathcal{S}_{F_{2}}$$
\end{theorem}

According to Definition \ref{decomposition property}, this means that $\mathbf F^{\circ}\left( A^{\ast}\right):= \mathbf F\left( A^{\ast }\right)\setminus \{\emptyset \}$ has the decomposition property. 

Among metric spaces over $\mathbf F(A^{*})$ are those coming from reflexive and involutive transition sytems. They are introduced in the next section.

\section{Transition systems as metric spaces}\label{section:trans}

Let $A$ be a set. 
A \textit{transition system} on the alphabet $A$ is a pair $M:=(Q$,
$T)$ where $T\subseteq Q\times A\times Q.$ The elements of $Q$ are
called \textit{states} and those of $T$ \textit{transitions}. Let
$M:=\left(Q,T\right) $ and $M^{\prime }:=\left( Q^{\prime
},T^{\prime }\right) $ be two transition systems on the alphabet
$A$. A map $f:Q\longrightarrow Q^{\prime }$ is a 
\textit{morphism} of transition systems if for every transition
$(p,a,q)\in T$, we have $(f(p),a
, f\left( q\right)) \in T^{\prime }$. When $f$ is bijective and $%
f^{-1} $ is a morphism from $M^{\prime }$ to
$M$, we say that $f$ is an \textit{isomorphism}.
The collection of transition systems over $A$, equipped with these morphisms,
form a category. This category has products. If $\left(
M_{i}\right) _{i\in I}$ is a family of transition systems,
$M_{i}:=\left( Q_{i},T_{i}\right),$ then their product $M$ is the
transition system $\left( Q,T\right) $ where $Q$ is the direct product $\underset{i\in I%
}{\prod }Q_{i}$ and $T$ is defined as follows: if $x:=\left(
x_{i}\right) _{i\in I}$ and $y:=\left( y_{i}\right) _{i\in I}$ are
two elements of $Q$ and $a$ is a letter, then $\left(
x,a,y\right) \in T$ if and only if
$\left( x_{i},a,y_{i}\right) \in T$ for every $i\in I$.  

An \textit{
automaton} $\mathcal A$ on the alphabet $A$ is given by a transition system
$M:=\left( Q,T\right)$ and two subsets $I,$ $F$ of $Q$ called the set
of \textit{initial} and \textit{final states}. We denote the
automaton as a triple $\left( M,I,F\right)$.  A \textit{path} in
the automaton $\mathcal{A}:=\left( M,I,F\right)$ is a sequence
$c:=\left( e_{i}\right) _{i<n}$ of consecutive
transitions, that is  of transitions $e_{i}:=(q_{i},a _{i},q_{i\cdot1})$.
The  word $\alpha:=a_{0}\cdots a_{n-1}$ is the \textit{label} of the path, the 
state $q_{0}$ is its \textit{origin} and the state $q_{n}$ its \textit{end}.
One agrees to define for each state $q$ in $Q$ a unique null path
of length $0$ with origin and end $q$. Its label is the empty word
$\Box$. A path is \textit{successful} if its origin is in $I$ and
its end is in $F$. Finally,  a word $\alpha$ on the alphabet $A$ is
\textit{accepted} by the automaton
$\mathcal{A}$ if it is the label of some successful path. The \textit{%
language accepted} by the automaton $\mathcal{A}$,  denoted by
$L_{\mathcal{A}}$,  is the set of all words accepted by
$\mathcal{A}$. Let $\mathcal{A}:=\left( M,I,F\right) $ and
$\mathcal{A}^{\prime }:= \left( M^{\prime },I^{\prime },F^{\prime
}\right) $ be two automata. A {\it morphism} from $\mathcal{A}$ to $\mathcal{A}'$ is a map $f:Q\longrightarrow Q^{\prime }
$ satisfying the two conditions:
\begin{enumerate}
\item  $f$ is  morphism from $M$ to $M^{\prime }$; 
\item  $f$ $(I)$ $\subseteq I^{\prime }$ and $f(F)\subseteq F^{\prime}$. 
\end{enumerate}
If, moreover, $f$ is bijective, $f(I)=I^{\prime },f(F)=F^{\prime }$ and $%
f^{-1}$ is also a  morphism from $\mathcal{A}^{\prime }$ to $%
\mathcal{A}$, we say that $f$ is an  \textit{isomorphism} and that
the two automata $\mathcal{A}$ and $\mathcal{A}^{\prime }$ are \textit{%
isomorphic}.

To a metric space $\left( E,d\right) $ over
$\underline{F}(A^{\ast })$, we may associate the transition system
$M:=\left( E,T\right) $ having $E$ as   set of states and
$T:=\left\{ \left( x,a,y\right) :a\in d\left( x,y\right) \cap
A\right\} $ as   set of transitions. Notice that such a
transition system has the following properties: for all $x,y\in E$
and every $a,b\in A$ with $b\geq a$:\\
1) $\left( x,a,x\right) \in T$; \\
2) $\left( x,a,y\right) \in T$ implies $\left(
y,\overline{a},x\right) \in T$; \\
3) $\left( x,a,y\right) \in T$ implies  $\left( x,b,y\right) \in
T.$\\
We say that a transition system satisfying these properties is \textit{%
reflexive} and \textit{involutive} (cf. \cite{saidane, kabil-pouzet}). Clearly \ if $%
M:=\left( Q,T\right) $ is such a transition system, the map
$d_{M}:Q\times Q\longrightarrow \mathbf F(A^{\ast })$ where $d_{M}\left(
x,y\right) $ is the language accepted by the automaton $\left(
M,\left\{ x\right\},\left\{ y\right\} \right) $ is a distance. We
have the following:

\begin{lemma}\label{lem:metric-transition}
Let $\left( E,d\right)$ be a metric space over
$\underline{F}\left( A^{\ast }\right)$.  The following properties are
equivalent: 
\begin{enumerate}[(i)]
\item The map $d$ is of the form $d_{M}$ for some
reflexive and involutive transition system $M:=(E,T)$; 
\item  For all $\alpha,\beta \in A^{\ast }$ and  $x$,
$y$ $\in E$, if $\alpha \cdot\beta \in d\left( x,y\right)$, then there
is some $z\in E$ such that $\alpha \in d\left( x,z\right)$ and
$\beta \in d\left( z,y\right)$.\end{enumerate}
\end{lemma}

The category of reflexive and involutive transition
systems with the morphisms defined above identify  to a subcategory of the
category having as objects the metric spaces and morphisms the contractions. Indeed:

\begin{fact}\label{fact:morphism}  Let $M_{i}:=\left( Q_{i},T_{i}\right) \left( i=1,2\right) $
be two reflexive and involutive transition systems. A map $%
f:Q_{1}\longrightarrow Q_{2}$ is a morphism from $M_{1}$ to $M_{2}$ if only
if $f$ is a contraction  from $(Q_{1},d_{M_{1}})$ to $%
(Q_{2},d_{M_{2}}).$
\end{fact}

From this fact, we can observe that if $\left( M_{i}\right) _{i\in I}$ is a
family of transition systems $M_{i}:=\left( Q_{i},T_{i}\right) $ then  the metric
space ($Q,$ $d)$ associated to the transition system $\underset{i\in I}{%
\prod }M_{i}$, product of the $M_{i}$'s, is the product of metric spaces ($%
Q_{i},$ $d_{i})$ associated to the transition systems $\left(
Q_{i},T_{i}\right)$.

Injective objects satisfy the convexity property stated in $(ii)$ of Lemma \ref{lem:metric-transition}. Hence, if $F$ is a final segment of $A^{\ast }$, the distance $d$ on the injective envelope  $\mathcal {S}_F$ coincide with the distance $d_F$ associated with the transition system 
$M_{F}:=\left( S_{F},T_{F}\right)$ where $$T_{F}:=\left\{ \left(
p,a,q\right) :a \in d\left( p,q\right) \cap
A\right\}. $$  We   denote by $\mathcal{A}_{F}$ the automaton $\left(
M_{F},\left\{ x\right\},\left\{ y\right\} \right)$, where $x:= A^{*}$ and $y:= F$.

From the existence of the injective envelope, we get:

\begin{theorem}\label{envelopeF}
For every $F\in \mathbf F(A^{\ast})$ there is an involutive and reflexive  transition system $M:=(Q,T)$, an initial state $x$ and a
final state $y$ such that the language accepted by the automaton
$\mathcal{A}=(M,\{x\},\{y\})$ is $F.$ Moreover, if $A$ is
well-quasi-ordered then we may choose $Q$ to be finite.
\end{theorem}

\begin{proof} Take $M:= M_F$,  $x:= A^{*}$ and $y:= F$. If  $A$ is well-quasi-ordered then $A^*$ is also well-quasi-ordered  (Higman\cite{higman}), hence the final
segment $F$ has a finite basis, that is, there are finitely many
words $\alpha _{0},...,\alpha _{n-1}$ such that $F=\{\alpha
:\alpha _{i}\leq \alpha\; \mbox{for some}\; i<n\}$. 
\end{proof}

Since injective objects in the category of metric spaces satisfy the convexity property stated in $(ii)$ of Lemma \ref{lem:metric-transition}, their distance is the  distance associated with a transition system. Thus we may reproduce Theorem \ref {thm:main2} almost verbatim. We get:

\begin{theorem}\label{thm:main3}
Let $M:= (Q, T)$ be a transition system, $X\subseteq Q$.  
Then $(M, d_M)$ is isomorphic to the injective envelope of $(X, d_{M\restriction X})$  iff for every reflexive and involutive transition system  $M^{\prime}:= (Q', T')$, $X'\subseteq Q'$, every isometry $f: (X, d_{M\restriction X})$ onto $(X', d_{M'\restriction X'})$:

\begin{enumerate}[{(i)}]
\item Every non-expansive map    $\overline f: (M, d_M)\rightarrow (M',d_{M'})$, if any,  which extends $f$ is  an isometric embedding of $(M, d_M)$ into $(M',d_{M'})$; 

\item The isometry $f^{-1}$ extends to a non-expansive map from  $(M', d_{M'})$ into $(M,d_M)$. \end{enumerate}\end{theorem}

Taking for $X$ a $2$-element subset  $\{x,y\}$ of $Q$, Theorem \ref {thm:main2} translates to Theorem \ref{thm:main1} stated in the introduction. 

\subsection{Minimal automaton and Minmax automata}
We suppose now that the alphabet $A$ is finite. 
We refer to \cite{sakarovitch} Subsection 3.3 p. 111-118 for the construction of the minimal state deterministic automaton.
If  $F\in \mathbf F(A^{\ast})$ and $u\in A^{\ast}$,  the \emph{left residual} is $u^{-1}F:= \{v\in A^{\ast}: uv\in F\}$.   The minimal automaton $Min_F$ recognizing $F$ has $Q_F:= \{u^{-1}F: u\in A^{\ast}\}$ as set of states. Its initial state is  $x:= F$,  its final state is $y:= A^{\ast}$ and  the transition function $\delta_F$  associate to the pair $(u^{-1}F, a)\in Q_F \times A$ the state $(u a)^{-1}F$. 

One can check that:

\begin{lemma} The map $i: Min_F\rightarrow \mathcal A_F$ defined by $i(Y):= \underset{y\in Y}{\bigcap} Fy^{-1}$  is a morphism of automata. 
\end{lemma}

\begin{proof}
Suppose that $(u^{-1}F,a,(ua)^{-1}F)$ is a transtion in  the automaton $Min_F$. We prove that $(i(u^{-1}F),a,i((ua)^{-1}F))$ is a transition in the automaton  $\mathcal A_F$. This is equivalent to $a\in d(i(u^{-1}F),i((ua)^{-1}F)$. The last condition amounts to  1) $ i(u^{-1}F)a \subseteq i((ua)^{-1}F)$ and 2) $ i((ua)^{-1}F \overline{a} \subseteq i(u^{-1}F)$. For 1), let $v\in i(u^{-1}F)$, we claim that $va\in i((ua)^{-1}F)$, that is for each word $w$, if $uaw\in F$, then $vaw\in F$. Since $v\in i(u^{-1}F)$, from $uaw\in F$, we have $vaw\in F$, as required. The inclusion 2) follows directly from the fact that $F$ is a final segment.\\
The  equalities    $i(F) =  A^{*}$ and $i(A^{*})=F$   are obvious.      
\end{proof}

In general, the transition system associated to $Min_F$ is neither reflexive nor involutive. 
Among all
involutive and reflexive transition systems $M:= (Q, T)$ with 
initial state $x$ and final state $y$ such that the language
accepted by the automaton $( M,\{ x\},\{ y\}) $ is $F$, we select
those with the least number of states and among those transition systems, we select
those with a maximum number of transitions; we call 
\textit{minmax transition systems} these transition systems. The automaton corresponding to
a minmax transition system  is
 a  \textit{minmax
automaton}. In other terms, an automaton
$\mathcal{A}:=(M,\{x\},\{y\})$ is minmax if $(a)$ it is reflexive and involutive,  $(b)$ for every reflexive
and involutive automaton $\mathcal{A^{\prime }}= ( M^{\prime },\{
x^{\prime }\},\{ y^{\prime }\}) $ such that $d_{M}( x,y)
=d_{M^{\prime }}( x^{\prime },y^{\prime }),$ one has $\left|
Q\right|   \leq \left| Q^{\prime }\right| $, moreover $(c)$ 
$\left| Q\right| =\left| Q^{\prime }\right| $ implies
$\left| T\right| \geq \left| T^{\prime }\right|$.
\begin{proposition}\label{prop:minmax} Let $M:=(Q,T)$ be transition system and $F$ be
a final segment of $A^{\ast }$. If $\mathcal{A}:= \left( M,\left\{
x\right\},\left\{ y\right\} \right) $ is a minmax automaton which
accepts $F$, then $\mathcal{A}$ is isomorphic to an induced
automaton $\mathcal{A}^{\prime}$ of $\mathcal{A}_{F},$ that is
there is some $Q'\subseteq \mathcal{S}_{F}$ containing $A^{*}$ and $F$ such that
the automaton $\mathcal{A}^{\prime}:=({M_{F}}_{\restriction Q'}, \{A^*\}, \{F\})$ is isomorphic to $\mathcal{A}$.
\end{proposition}
\begin{proof}
 To the automaton $\mathcal{A}:=\left( M,\left\{ x\right\},\left\{
y\right\} \right),$ we associate the metric space $\left(
Q,d_{M}\right) $ induced by $M$. Identifying $x$ with $A^{*}$ and $y$
with $F = d_{M}\left( x,y\right),$ the space $\mathcal{S}_{F}$
identifies with the injective envelope of the $2$-element space $\left\{ x,y\right\}.$
Let $i$ be the identity mapping on $\left\{ x,y\right\} $. This is
a partial non-expansive mapping from $\left( Q,d_{M}\right) $ into
$\mathcal{S}_{F}$. Since $\mathcal{S}_{F}$\ is  injective, this
partial map extends to a non-expansive mapping $f$ defined on $Q$.
Since $f$ is non-expansive, this map is an automata morphism. Let
$Q^{\prime }:=f\left( Q\right) $ and consider the transition system
$M^{\prime }:=\left( Q^{\prime },\mathcal S_{F}\restriction Q^{\prime }\right)$. We
have $F=d\left( x,y\right) \leq
d_{M^{\prime }}\left( x,y\right) \leq d_{M}\left( x,y\right) =F.$ Thus $%
d_{M^{\prime }}\left( x,y\right) =F.$ Since $\vert Q\vert$\ is minimal, it
follows that $\left| Q\right| =\left| Q^{\prime }\right| $ proving
that $f$ is injective. Now since $\left| T\right| $ is maximum, we
have $\left| T_{F\restriction Q^{\prime }}\right| =\ \left| T\right| $
proving that $f$ is an automata isomorphism from $\mathcal A$ to $\mathcal A^{\prime
}:=\left( M^{\prime },\left\{ x\right\} ,\left\{ y\right\} \right)$.
\end{proof}
 Let $M_{1}:=(Q_{1},T_{1})$ and $M_{2}:=(Q_{2},T_{2})$ be
two transition systems such that $Q_{1}$ and $Q_{2}$ have only the element $%
y $ in common. Let $\mathcal{A}_{1}:=\left( M_{1},\left\{
x\right\},\left\{ y\right\} \right) $ and $\mathcal{A}_{2}:=\left(
M_{2},\left\{ x\right\} ,\left\{ z\right\} \right) $ be two
automata. We denote $\mathcal{A}_{1}\cdot\mathcal{A}_{2}$ the
automaton $\left(
M,\left\{ x\right\},\left\{ y\right\} \right) $ where $M=(Q,T)$ $%
Q=Q_{1}\cup Q_{2}$ and $T=T_{1}\cup T_{2}.$
\begin{theorem} Let $F$ be a non-empty final segment of $A^{\ast }$. If  $F=F_{1}F_{2}$,with $F_{1}$ and $F_{2}$ final segments
of $A^{\ast },$ then an automaton $\mathcal{A}$ with initial
state $x$ and  final state $y$ accepting $F$
is minmax if and only if it decomposes into $\mathcal{A}_{1}\cdot\mathcal{A}_{2}$
 where $\mathcal{A}_{1}$ and $%
\mathcal{A}_{2}$ are two minmax automata accepting respectively
$F_{1}$ between $x$  and  $z$ and $F_{2}$ between $z$ and  $y$.
\end{theorem}
\begin{proof}  From Theorem \ref{thm:sum}, we have $\mathcal{
S}_{F_{1}}\cdot\mathcal{S}_{F_{2}}\cong \mathcal{S}_{F_{1}F_{2}}. $
According to Proposition \ref{prop:minmax}, a minmax automaton $\mathcal{A}$
accepting $F$ is a subautomaton of
$\mathcal{A}_{F}=\mathcal{A}_{F_{1}}\cdot\mathcal{A}_{F_{2}}.$
Thus $\mathcal{A}$ decomposes into $\mathcal{A}_{1}\cdot\mathcal{A}_{2}$ where
$\mathcal{A}_{1}$ is a subautomaton of $%
\mathcal{A}_{F_{1}}$and $\mathcal{A}_{2}$ is a subautomaton of
$\mathcal{A}_{F_{2}}.$ Since $\mathcal{A}$ is minmax, both
$\mathcal{A}_{1}$ and $\mathcal{A}_{2}$ are minmax. Conversely,
assume that $\mathcal{A}_{1}$ and $\mathcal{ A}_{2}$ are minmax.
Let $\mathcal{ A} $ be a minmax automaton accepting $F_{1}\cdot F_{2}.$
Then $\mathcal{ A} $ decomposes into $\mathcal{ A}_{1}^{\prime
}\cdot\mathcal{A}_{2}^{\prime }$ where $\mathcal{A}_{1}^{\prime }$ and
$\mathcal{A}_{2}^{\prime } $ are minmax. The automaton
$\mathcal{A}_{1}$ and $\mathcal{A}_{1}^{\prime } $ (resp.
$\mathcal{A}_{2}$ and $\mathcal{A}_{2}^{\prime }$) have the same
number of states and transitions. That is
$\mathcal{A}_{1}\cdot\mathcal{A}_{2}$ is minmax.
\end{proof}
\begin{example} (Mike Main,1989, communicated by Maurice
Nivat \cite{nivat}). We give an example of two non-isomorphic minmax automata
accepting the same language $L\in F\left( A^{\ast }\right) $.

Let $A=\left\{ a,b,c,a^{\prime },b^{\prime },c^{\prime }\right\} $ with $%
\overline{a}=a^{\prime },\overline{b}=b^{\prime }$ and $\overline{c}%
=c^{\prime }.$ Consider the automata represented on Figure 2. To
each of these  
automata, we associate the involutive and reflexive automata obtained by replacing each transition ($p$%
, $\alpha, q$) by ($p$, $\alpha, q$) and ($q$, $\overline{\alpha
},p$) and adding a loop at every vertex.The language accepted by each of these automata 
between $x$ and $y$ is $L=\uparrow \left\{ab,ac,ba,bc,ca,cb\right\} .$ As it is easy to check, these automata are minmax but not isomorphic.
\end{example}
\unitlength=1cm 
%Premiere figure
\begin{picture}(16,7.5) 
\put(0,4){\circle*{.15}}
\put(3,4){\circle*{.15}}
\put(6,4){\circle*{.15}}
\put(3,7){\circle*{.15}}
\put(3,1){\circle*{.15}} 
\put(0,4){\line(1,0){3}}
\put(1.5,4){\vector(1,0){.1}}
\put(1.5,4.2){{\small $b$}} 
\bezier{300}(0,4)(1,7)(3,7)
\put(1,6){\vector(1,1){.1}}
\put(.7,6){{\small  $a$}}
\bezier{300}(0,4)(1,.9)(3,1)
\put(1,2){\vector(1,-1){.1}}
\put(1,2.2){{\small $c$}} 
\bezier{300}(3,7)(5,7.1)(6,4)
\put(5,6){\vector(1,-1){.1}} 
\put(5,6.2){{\small $b$}} 
\bezier{300}(6,4)(5,.9)(3,1)
\put(5,2){\vector(1,1){.1}}
\put(4.8,2) {{\small $a$}}
\bezier{200}(3,4)(4.5,5)(6,4)
\put(4.5,4.5){\vector(1,0){.1}}
\put(4.4,4.2){{\small $a$}}
\bezier{200}(3,4)(4.5,3)(6,4)
\put(4.5,3.5){\vector(1,0){.1}} 
\put(4.5,3.7){{\small $c$}} 
\bezier{300}(3,7)(4,4)(6,4)
\put(4,5){\vector(1,-1){.1}}
\put(4,5.2){{\small $c$}} 
\bezier{300}(6,4)(4,4)(3,1)
\put(4,3){\vector(1,1){.1}} 
\put(3.7,3){{\small $b$}}
\put(-.4,4){$x$}
\put(6.2,4){$y$} 
%Fin de premiere figure
%
%Deuxieme figure

\put(8,4){\circle*{.15}}
\put(11,4){\circle*{.15}}
\put(14,4){\circle*{.15}}
\put(11,7){\circle*{.15}}
\put(11,1){\circle*{.15}} 
\put(11,4){\line(1,0){3}}
\put(12.5,4){\vector(1,0){.1}}
\put(12.5,4.2){{\small $b$}} 
\bezier{300}(8,4)(9,7)(11,7)
\put(9,6){\vector(1,1){.1}}
\put(8.7,6){{\small  $b$}}
\bezier{300}(8,4)(9,.9)(11,1)
\put(9,2){\vector(1,-1){.1}}
\put(9,2.2){{\small $a$}}
\bezier{300}(11,7)(13,7.1)(14,4)
\put(13,6){\vector(1,-1){.1}} 
\put(13,6.2){{\small $a$}}
\bezier{300}(14,4)(13,.9)(11,1)
\put(13,2){\vector(1,1){.1}}
\put(12.8,2) {{\small $c$}}
\bezier{200}(8,4)(9.5,5)(11,4)
\put(9.5,4.5){\vector(1,0){.1}}
\put(9.4,4.2){{\small $a$}}
\bezier{100}(8,4)(9.5,3)(11,4)
\put(9.5,3.5){\vector(1,0){.1}} 
\put(9.5,3.7){{\small $c$}}
\bezier{300}(8,4)(10,3.9)(11,7)
\put(10,5){\vector(1,1){.1}}
\put(9.8,5){{\small $c$}} 
\bezier{300}(8,4)(10,4.1)(11,1)
\put(10,3){\vector(1,-1){.1}} 
\put(10.2,3){{\small $b$}} 
\put(7.6,4){$x$}
\put(14.2,4){$y$}
\put(6.7,0){Figure 2}
\end{picture}
\section{A description of the injective envelope}

We describe the injective envelope $\mathcal{S}_{F}$ in 
Galois lattice terms. We derive a test of its finiteness which leads to the notion of Ferrers language.  

\subsection{Incidence structures}

We recall first some basic facts about incidence structures, Galois lattices and Ferrers relations. We follow the exposition given in \cite{pouzet-sikaddour-zaguia}.

An
\textit{incidence structure} $R$ is a triple $\left(V,\rho
,W\right) $ where $\rho $ is a subset of the product $V\times
W.$ We set  $\rho^{-1}:=\{(x,y): (y,x)\in \rho\}$ and $R^{-1}:= (W, \rho^{-1}, V)$, that we call the \emph{dual} of $R$. We denote by  $\lnot \rho $ the relation $V\times W\setminus \rho $ and $%
\lnot R$ the resulting incidence structure.\\
A subset of $V\times W$ of the form $X\times Y$ is a \emph{rectangle}. Let $\left( X,Y\right)\in \powerset(V)\times \powerset(W)$. We set 
$R_{\wedge}( X):=\{y \in W: x\rho y\ \mbox{for all} \; x\in X\}$, 
$R^{-1}_{\wedge}( Y):= \{x\in V: x\rho y\ \mbox{for all}\ y\in Y\}.$
We recall that  the set $X\times Y$ is a maximal rectangle included into $\rho$ if and only if $X= R^{-1}_{\wedge}(Y)$ and $Y=R_{\wedge}( X)$.   
The \textit{Galois lattice} $\gal(R) $ of $R$ is the collection,
ordered by inclusion, of subsets of $V$ of the form $R^{-1}_{\wedge}( Y)$ for $Y\in \powerset(W)$. This is  a complete lattice; the largest element is $V$ ($=R^{-1}_{\wedge}(\emptyset)$). Then,  $\gal (R^{-1})$ is the collection,  ordered by inclusion,   of subsets of $W$ of the form $R_{\wedge}( X)$ for $X\in \powerset(V)$. We recall the important fact that $\gal (R)$ and $\gal (R^{-1})$ are dually isomorphic. Since  $\gal(R) $ consists of   intersections of sets of the form $R^{-1}(y)$ for $y\in W$, $\gal (R)$ is finite if and only if the set of $R^{-1}(y)$ for $y\in W$ is finite; since $\gal(R)$ is dually isomorphic to $\gal(R^{-1})$, it is finite if and only if the set of $R(x)$ for $x\in V$ is finite. 

We give two examples from the theory of ordered sets. 

\begin{fact}\label{fact7}   If $R:=\left( P,\leq, P\right)$ is an ordered set, 
then $\gal (R)$,  the \emph{MacNeille completion of} $P$,  is a  complete  lattice in which every member is a join and a meet of elements of $P$. And $\gal(\lnot R)$ is the set of final segments of $P$ ordered by inclusion.
\end{fact}

Now, we mention the facts we need. 

Let $R$ := $\left( V,\rho,W\right) $ and $R^{\prime }:=\left(
V^{\prime },\rho ^{\prime },W^{\prime }\right) $ be two incidence
structures. According to Bouchet \cite{Bo}, a \textit{coding} from $R$ into $%
R^{\prime }$ is a pair of maps $f:V\longrightarrow V^{\prime
},g:W\longrightarrow W^{\prime }$ such that
$$x\rho
y\Longleftrightarrow f\left( x\right) \rho ^{\prime }g\left(
y\right).$$

Bouchet's  Coding Theorem \cite{Bo} below is a striking illustration of the links between coding and embedding.
\begin{theorem}\label{bouchethm} Let $(T ,\leq)$ be a complete lattice and $R$ be an incidence structure. Then $R$ has a coding into
$(T ,\leq, T )$ if and only if $\gal(R)$  is embeddable in $T$.
\end{theorem}
%Corollary 2. Let R := (E, ?, F ) and R? := (E?, ??, F ?) be two incidence structures. Then Gal(R) is
%embeddable in Gal(R?) whenever R has a coding into R? .
The basic facts about coding we need are the following:

\begin{fact}\label{fact9} Let $\left(f,g\right) $ be a coding from $R$ into $%
R^{\prime }$.
\begin{enumerate}[(a)]
\item $\left( f,g\right) $ is a coding from $\lnot R$ to $\lnot R^{\prime };$
\item $\gal(R)$ is embeddable into $\gal(R')$; 
\item If $f$ is surjective, then $\gal(R)$ identifies with an intersection closed
subset of $\gal( R^{\prime })$.
\end{enumerate}
\end{fact}
\begin{proof}
$(a)$ Immediate consequence of the definition.

\noindent$(b)$ Follows from Bouchet's Theorem. 

\noindent $(c)$ The map $X\longrightarrow X^{\prime }=R'^{-1}_{\wedge}
(R'_{\wedge}( X)) $ is an embedding
from $\gal(R) $ into $\gal(R^{\prime }) $ which
preserves non-empty intersections. If $f$ is surjective, then the
least element of the Galois lattice is preserved, hence $\gal(
R)$  identifies with an intersection closed subset  of
$\gal(R^{\prime}).$
\end{proof}

\begin{fact}\label{fact10}  Let $\left( f_{i},g_{i}\right)_i$ be a family of codings
from $\left(V,\rho _{i},W\right)$ into $\left( V_{i},\theta
_{i},W_{i}\right)$. Then, the pair $\left( \Pi_i f_{i},\Pi_i g_{i}\right) $ is a coding from $\left(
V,\bigcup_i \rho _{i},W\right) $ into $\left( \Pi_i V_{i},\lnot \Pi_i \lnot \theta
_{i},\Pi_i W_{i}\right).$
\end{fact}

\begin{proof}
Clearly $( \Pi_i f_{i},\Pi_i g_{i}) $ is a coding from $(V,\bigcap_i \rho _{i},W)$ into 
$(\Pi_i V_{i},\Pi_i \theta _{i},\Pi_i W_{i})$.  
Indeed, if  $x:=( x_{i})_i\;  \text{and}\;  y:=(y_{i})_i,\; \text{then }\;  x\Pi_i \theta _{i}y\;  \text{ means }\; x_{i}\theta_i y_{i} \; 
\text{ for all }i$. From Fact \ref{fact9} (a) we get that $\left( \Pi_i f_{i},\Pi_i
g_{i}\right) $ is a coding from $\left( V,\bigcup_i \rho _{i},W\right) $ into $%
\left( \Pi_iV_{i},\lnot \Pi_i \lnot \theta _{i},\Pi_i W_{i}\right) .$
\end{proof}

\begin{corollary}\label{cor10}If for every $i$, $V_{i}=W_{i}$ and $\theta _{i}$
is of the form $\lnot \leq_{i} $ for some ordering $\leq_{i} $ on $V_{i}$,
then there is a coding from $\left( V,\bigcup_i \rho _{i},W\right) $ into
$\left( \Pi_i V_{i},\lnot \leq,\Pi_i V_{i}\right) $ where $\leq $ is
the product ordering on  $\Pi_i V_{i}$.
\end{corollary}

\begin{fact}\label{fact11}  Let $R:=\left( V,\rho,W\right)$ be an incidence
structure. Then $\gal(R)$ is finite if and only if $\gal(\lnot
R) $ is finite.
\end{fact}

\begin{proof}
 For each $y\in V$, $(\lnot R)^{-1}(y)
=V\setminus R^{-1}(y)$. Since
$\gal(R) $ is made of intersections of sets of the form
$R^{-1}(y)$,  if $\gal (R)
$ is finite, the
collection of such sets is finite, hence the collection of sets of the form
$(\lnot R)^{-1}(y)$
 is finite too. Since $
\gal (\lnot R) $ is made of intersections of   these sets, it is
finite.
\end{proof}

\begin{fact}\label{fact12}  Let $R_{i}$ := $\left( V_{i},\rho _{i},W_{i}\right) $ be a
family of incidence structures. Then $\gal( \Pi_i R_{i}) $ embeds into
$\Pi_{i} \gal(R_{i}).$

\end{fact}

\begin{proof}
Let $y:=\left( y_{i}\right)_i \in \Pi_i W_{i}.$ We have $(\Pi_i R_
{i})^{-1} (y) =\Pi_i R^{-1}_i( y_{i})$. Hence, if
$p_{j}:\Pi_i V_{i}\longrightarrow V_{j}$ denotes the $j^{th}$
projection, then for each $X\in \gal (\Pi_i R_{i}),$ we have
$\left( p_{i}\left( X\right) \right) _{i}\in \Pi_i\gal(
R_{i}).$ This defines an embedding, proving our claim. If
each member $X$ of $\gal (\Pi_i R_{i}) $ is non-empty, then
this embedding is an isomorphism.
\end{proof}

\begin{fact}\label{13}  Suppose $V_{i}=V,W_{i}=W$ for each $i\in I$ and $I$
finite. If $\gal(R_{i})$ is finite for each $i\in I,$ then $\gal ((  
V,\bigcup_i \rho _{i},W  ) )$ is finite.
\end{fact}

\begin{proof}
\ According to Fact \ref{fact11}, this amounts to prove that $\gal \left(\left(
V,\bigcap_i \lnot \rho _{i},W\right) \right) $ is finite. 
%Since each $(V, \rho_i, W)$ has a coding into $(V, \rho_i, E)$ (namely, the pairs of identity maps), 
Fact \ref{fact10} yields that $\left( V,\bigcup  \rho _{i},W\right)$ has a coding into $(V^V,  \lnot \Pi_i \lnot\rho
_{i},W^W)$. According to $(a)$ of Fact \ref{fact9}, $\left(V,\cap  \lnot \rho _{i},W\right)$ has a coding into $(V^V,\Pi_i\lnot\rho
_{i},W^W)\equiv \Pi_i   ( V,\lnot \rho _{i},W)$.  According to $(b)$ of 
Fact \ref{fact9}, $\gal \left(
\left( V,\bigcap_i \lnot \rho _{i},W\right) \right) $ embeds into
$\gal \left( \Pi_i \ \left( V,\lnot \rho _{i},W\right) \right),$which in
turns embeds into $\Pi_i \gal\left( V,\lnot \rho _{i},W\right) $ from
Fact \ref{fact12}. From Fact \ref{fact11}, each $\gal \left( V,\lnot \rho _{i},W\right)$ is
finite. The result follows.
\end{proof}

\subsection{Languages and their Galois lattices}
 Galois lattices arose from group theory and geometry. We show here how they interact with language theory.
 
 We represent subsets of the free monoid  by incidence structures as follows. 

Let $A$ be an alphabet and  $L$ be a subset of $A^{\ast }$. Denoting by $xy$ the concatenation of the words $x, y \in A^*$,  define a  binary  relation  $\rho
_{L}$ on $A^{\ast }$ by:

            $$x\rho _{L}y\Longleftrightarrow xy\in L$$

and  set $R_{L}:= (A^*, \rho_{L}, A^*)$.

We mention without proof some properties of this  association.

It  preserves Boolean operations,  that is:
\begin{fact}\label{fact:complement}
\begin{enumerate}
\item $\rho_{A^*\setminus L}= A^*\times A^*\setminus \rho_L$. Hence, $R_{A^*\setminus L}= \neg R_{L}$. 

\item $\rho_{\bigcup_i L_i}= \bigcup_{i} \rho_{L_i}$ and  $\rho_{\bigcap_i L_i}= \bigcap_{i} \rho_{L_i}$
\end{enumerate}
\end{fact}

Not every binary relation $\rho$ on $A^*$ can be  of the form $\rho_L$. For an example, if $\rho$ contains a pair $(u, v)$ with $u,v\in A^*$, then $\rho$ contains  $\rho_{\{w\}}$  where $w:=uv$. In fact, if $\rho$ is a binary relation on $A^*$, set $\pi_{\rho}:= \{uv: (u,v)\in \rho\}$; then $\rho_{\pi_{\rho}}$
is the least binary relation of the form $\rho_L$ containing $\rho$. In particular, $\rho_{\{w\}}$ is the least relation of the form $\rho_L$ containing a pair $(u,v)$ such that $w= uv$.   

In the Boolean lattice made of relations of the form $\rho_L$, the atoms are of the form $\rho_{\{w\}}$ where   $w$ is any word in $ A^{*}$. They form a partition  of $A^*\times A^*$. Every binary relation  of the form $\rho_L$ is an union of some blocks of this partition. 

We will say more about this association in Section \ref{ferrers relation}.

We conclude this subsection by the following property. 

 If $B$ is a subset of $A^*$,  we set $R_{L}{\restriction  B}:=(B,\rho_L \cap B\times B, B)$.

\begin{fact}\label{fact8}  The Galois lattices of  $R_{L}$ and and $R_{L}{\restriction  A^{\ast }\setminus L}$  are isomorphic
provided that $L$ is a final segment of $A^{\ast }.$
\end{fact}

\begin{proof}
Observe \ that for every subset $X$ of $A^{\ast }$, we have:
$R^{-1}_{\wedge}
(R_{\wedge}( X)) =R^{-1}_{\wedge}
(R_{\wedge}( X\setminus L)).$
\end{proof}

\vspace{.5cm}

\subsection{Proof of the first part of Theorem \ref{injenv}}

 From Corollary \ref {cor10}, and $c)$ of Fact \ref{fact9},  in order to prove the first part of 
Theorem \ref{injenv}, namely that $\mathcal S_F$ can be identified to an intersection closed subset of  the set
$\mathbf F(\mathbf{n}_{0}\otimes \cdots \otimes \mathbf{n}_{k-1})$,  all  that we need is to prove that for each $i$, there is a
coding $\left( f_{i},g_{i}\right),$ with $f_{i}$ surjective, from
$\left( A^{\ast },\rho
_{Fi},A^{\ast }\right) $ into $\left( \mathbf{n}_{i},\lnot \leq,\mathbf{n}%
_{i}\right)$. But, this is false. From Fact \ref{fact8}, a coding from $\left( A^{\ast
}\setminus F,\rho _{Fi},A^{\ast }\setminus F\right) $ into $\left( \mathbf{n}%
_{i},\lnot \leq,\mathbf{n}_{i}\right) $ suffices. Let
$X_{0},...,X_{n-1}$ be non-empty subsets of $A$ and let 
$X:=X_{0}\cdot\cdot \cdot X_{n-1}$
(that is the set of words $x_{0}\cdots x_{n-1}$ with $x_{i}\in X_{i})$ and let 
$F:=\uparrow X$.  Let $\mathbf{n}:=\left\{ 0,...,n-1\right\}$ be equipped with
the natural ordering. Let $f:A^{\ast }\setminus F\longrightarrow $ $\mathbf{n%
}$ and $g:A^{\ast }\setminus F\longrightarrow $ $\mathbf{n}$
defined as follows: For $v\in A^{\ast }\setminus F,$ if there is
$m\in \N$ such that $v\in \uparrow (
X_{0}\cdot\cdot \cdot X_{m-1})$ then $f\left( v\right) $ is the
largest $m$ having this property, otherwise $f(v)=0$.  For $w\in
A^{\ast }\setminus F$,  if there is $p\in \N$ such that $v\in
\uparrow \left\{ X_{p-1}\cdot \cdot \cdot X_{n-1}\right\}$ then $g\left(
v\right) $ is the least $p$ having this property, otherwise
$g(v)=n-1.$
By a straightforward verification, we have the following:

\begin{fact} \label{fact14} The map $f$ is surjective and the pair $(f,g)$
form a coding from\\
 $\left( A^{\ast }\setminus F,\rho
_{F},A^{\ast }\setminus F\right) $ into $\left( \mathbf{n},\lnot
\leq,\mathbf{n}\right).$
\end{fact}

This proves the first part of  Theorem \ref{injenv}.

\subsection{Proof of the second part of  Theorem \ref{injenv}} 

Let $F:=F_{0}\cup ...\cup F_{k-1}$,  let $X_{i}\subseteq F$ such that $%
F_{i}=\uparrow X_{i}$ and let $X_{i_{0}},...,X_{i_{n_{i}-1}}\subseteq A$
such that $X_{i_{0}}\cdot\cdots \cdot X_{i_{n_{i}-1}}=X_{i}.$ From Fact \ref{fact14} above, for
each $i,$ we have a coding $(f_{i},g_{i})$ from $\left( A^{\ast }\setminus
F_{i},\rho _{Fi},A^{\ast }\setminus F_{i}\right) $ into $\left( \mathbf{n}%
_{i},\lnot \leq,\mathbf{n}_{i}\right) $. The restrictions of $f_{i}$ and  $%
g_{i}$ to $A^{\ast }\setminus F$ give a coding that we still denote $%
(f_{i},g_{i})$ from $\left( A^{\ast }\setminus F_{i},\rho _{Fi},A^{\ast
}\setminus F_{i}\right) $ into $\left( \mathbf{n}_{i},\lnot \leq,\mathbf{n}%
_{i}\right) $ and $f_{i}$ is still surjective. From Fact \ref{fact10}, $\left( \Pi
f_{i},\Pi g_{i}\right) $ is a coding from $R_{F}:=\left( A^{\ast }\setminus
F,\cup \rho _{Fi},A^{\ast }\setminus F\right) $ into $R^{\prime }:=\left(
\mathbf{n}_{0}\otimes ...\otimes \mathbf{n}_{k-1},\lnot \leq,\mathbf{n}%
_{0}\otimes ...\otimes \mathbf{n}_{k-1}\right) $ where $\leq $ is the
natural ordering on the direct product $\mathbf{n}_{0}\otimes ...\otimes
\mathbf{n}_{k-1}.$ Since for each $i,f_{i}$ is surjective,  $\Pi $ $f_{i}$ is
surjective and from $(c)$ of Fact \ref{fact9}, $\gal \left( R_{F}\right) $ is identified with an
intersection closed subsets of $\gal \left( R^{\prime }\right) $ which is $%
\mathbf F\left( \mathbf{n}_{0}\otimes ...\otimes \mathbf{n}_{k-1}\right) $
by Fact \ref{fact7}. $\Box$
\subsection{An explicit isomorphism} For reader's convenience, let us
describe explicitly the
isomorphism between $\mathcal{S}_{F}$ and an intersection closed subset  of $%
\mathbf F\left( \mathbf{n}_{0}\otimes ...\otimes \mathbf{n}_{k-1}\right)
.$ For sake of simplicity, let us suppose that $F$ is finitely generated (which is the case if $A$ is w.q.o.).  Let $u_0, \dots, u_{k-1}\in A^*$, $n_0:= \vert u_0\vert, \dots, n_{k-1}:= \vert u_{k-1}\vert$, and  $F:= \uparrow \left\{
u_{0},...,u_{k-1}\right\} .$ Let $v\in A^{\ast };$ for each $i\in
\left\{ 0,...,k-1\right\},$ let $v''_{i}$ be the largest suffix of
$u_{i}$ such that $v''_{i}$ $\leq v$ and let $v_{i}^{\prime }$ be
the unique prefix of $u_{i}$ such that $u_{i}=v_{i}^{\prime }$
$v''_{i}.$ Set $s(v):=\{x:=\left(
x_{0},...,x_{k-1}\right) \in \mathbf{n}_{0}\otimes ...\otimes \mathbf{n}%
_{k-1}$ such that $x_{i}<\left| v_{i}\right| $ for all $i\}$ and $\tau
\left( v\right) $ $:=$ $\mathbf{n}_{0}\otimes ...\otimes \mathbf{n}%
_{k-1}\setminus $ $s(v).$\\
Define $\varphi :\mathcal{S}_{F}\longrightarrow \mathbf F\left( \mathbf{n}%
_{0}\otimes ...\otimes \mathbf{n}_{k-1}\right) $ by setting

            $$\varphi \left( X\right) :=\underset{v\in
A^{\ast }, X\subseteq \rho ^{-1}(v)}{\bigcap \tau \left( v\right)
}.$$\\
For $w:=w_{0}...w_{m}\in A^{\ast }$ and $\ell <\left| w\right|,$ let $%
w_{\mid \ell }=w_{0}...w_{\ell -1}$ be the restriction of $w$ to the first $%
\ell $ letters. Let $x:=\left( x_{0},...,x_{k-1}\right) \in $ $\mathbf{n}%
_{0}\otimes ...\otimes \mathbf{n}_{k-1},$ let $\mu \left(
x\right)=\uparrow \left\{ u_{i\mid x_{i-1}}:0\leq i\leq k-1\right\} $ be the final segment of $%
A^{\ast }$ generated by words of the form $u_{i\mid x_{i-1}}.$

Define $\Psi :\mathbf F\left( \mathbf{n}_{0}\otimes ...\otimes \mathbf{n}%
_{k-1}\right) \longrightarrow \mathcal{S}_{F}$ \ by setting

$$         \Psi \left( Y\right) :=\underset{x\in \mathbf{n}%
_{0}\otimes ...\otimes \mathbf{n}_{k-1}\setminus Y}{\bigcap \mu
\left( v\right) }$$\\
The maps $\varphi $ and $\Psi $ preserve intersections and $\Psi
\circ \varphi =id_{\mathcal S_{F}}.$

\subsection{Finiteness of the injective envelope}

\begin{proposition}\label{prop:finitenesswords}
Let $F$ be a final segment of $A^{\ast }$. The following conditions
are equivalent:

\begin{enumerate} [(i)]
\item  The injective envelope $\mathcal{S}_{F}$ is
finite; 
\item  $F$ is finite union $F=F_{0}\cup ...\cup
F_{i}\cup ...\cup F_{k-1}$ of final segments, each $F_{i}$
generated by a set $X_{i}$
of words $u_{i}$ of the same length $n_{i}$ all of the  form $%
u_{i}=u_{i_{0}}...u_{i_{j}}...u_{i_{n_{i}-1}}$ with $u_{i_{j}}\in
X_{i_{j}}\subseteq A.$

\end{enumerate}
\end{proposition}

\begin{proof}
$(i) \Rightarrow (ii)$. Let us suppose that $%
\mathcal{S}_{F}$ is finite. Let $\mathcal{A}_{F}$ be the corresponding reflexive and involutive 
automaton with an initial state $x:=$ $A^{\ast }$and a final state $y:=F$. Let
$Q^{*}_{F}$ be the set of finite sequences $s:=\left( s_{0},...,s_{n}\right) $
such that all $s_{i}$ are distinct states, $s_{0}:=x,s_{n}:=y,\left(
s_{i},a_i,s_{i+1}\right) \in T_{F}$ \ for some letter $a_i$. Since $\mathcal{S}_{F}
$ is finite this set is finite too. For each $s=\left(
s_{0},...,s_{n}\right) $ let $X_{s_{j}}:=\left\{ a\in A:\left(
s_{j},a,s_{j+1}\right) \in T_{F}\right\} $ and let $X_{s}:=%
\{a_{0}...a_{n-1}:a_{j}\in X_{s_j}\}$ and let $ F_{s}:=\uparrow
X_{s}$. It is easy to check that $F=\underset{s\in
Q_{F}}{\bigcup }F_{s},$ hence has the form mentioned above. 

\noindent $(ii) \Rightarrow (i)$. Apply Theorem \ref{injenv}.
\end{proof}

\begin{theorem}\label{thm:w.q.o.}
Let $\mathcal H :=\mathbf F(A^{\ast })$ be the Heyting algebra made of final segments of $%
A^{\ast }.$ The following conditions are equivalent:
\begin{enumerate}[(i)]
\item $A$ is well-quasi-ordered;
\item The injective envelope $\mathcal{N}\left(
E\right) $ of every finite metric space $E$ is finite.
\end{enumerate}
\end{theorem}

\begin{proof}
$\lnot  (i)\Longrightarrow \lnot (ii)$.  Let $F$ be
a final segment of $A^{\ast }$.  According to Proposition \ref{prop:finitenesswords}, the injective
envelope $\mathcal{S}_{F}$ is infinite whenever for each decomposition $%
F=\cup F_{i},$ some $F_{i}$ cannot be generated by a set of words having a
bounded length. This is the case if $F$ is generated by an infinite
antichain $X$ made of words of unbounded length. If $A$ is not w.q.o., then there
is an infinite bad sequence of letters, say $a_{0},...,a_{n},....$ The set $%
X=\left\{
a_{0},a_{1}a_{2},a_{3}a_{4}a_{5},a_{6}a_{7}a_{8}a_{9},...\right\}
$ is such an example.\\
$(i) \Longrightarrow (ii)$.  According to Proposition \ref{prop:finiteness}, it suffices
to show that the injective envelope of a two-element metric space
is finite. Let $F$ be a final segment of $A^{\ast }$. If $A$ is
w.q.o., $F$ satisfies condition $(ii) $ of Proposition \ref{prop:finitenesswords},
hence $\mathcal{S}_{F}$ is finite. 
\end{proof}

\noindent {\bf Comments.}
Proposition \ref{prop:finitenesswords} and Theorem \ref{thm:w.q.o.} are special instances of two basic
facts of language theory concerning rational languages, namely Myhill-Nerode Theorem (see \cite{sakarovitch} Theorem 2.3, p.247) and  a result of Ehrenfeucht-Haussler-Rozenberg (Theorem 3.3 of \cite{ehrenfeucht-haussler-rozenberg}, see \cite{sakarovitch} Theorem 5.3, p.296). For brevity, let us put these results together  in the context  of monoids. Let us recall that if $L$ is a subset of a monoid $M$, the \emph{residuals} of $L$ are sets of the form $Lv^{-1}:= \{u\in M: uv\in L\}$ for $v\in M$ and of the form $u^{-1} L:= \{v\in M: uv\in L\}$ for $u\in M$. The \emph{syntactic congruence} of $L$  is the largest congruence $\equiv_L$ on $M$ for which $L$ is an union of classes.  The set $L$  is \emph{recognisable} if it is  the union of classes of a congruence on $M$ which has finitely many classes(alias finite  index). The aforementioned results read as follows:
\begin{theorem}\label{Myhill-Nerode-Ehrenfeucht}
For a  subset $L$ of a monoid $M$, the following properties are equivalent:

\begin{enumerate}[(i)]
\item $L$ is recognisable; 
\item The set of residuals $\{Lv^{-1}:v\in M\}$ is finite;
\item The set of residuals $\{u^{-1}L:u\in M\}$ is finite; 
\item The syntactic congruence $\equiv_L$ has finite index;
\item $L$ is a final segment of a well-quasi-ordered set on $M$ which is compatible with the monoid operation. 
\end{enumerate}
\end{theorem}
  
  The equivalences from $(i)$ to $(iv)$ is Myhill-Nerode Theorem. The equivalence with $(v)$  is Ehrenfeucht-Haussler-Rozenberg Theorem.

            Let   $R:= (M, \rho_{L}, M)$ be the incidence structure where $\rho_L:= \{ (u,v) \in M\times M: uv\in L\}$. Then $u\rho _{L}v\Longleftrightarrow uv\in L$ and 
$R( u)=u^{-1}L$,   
 $R^{-1}(v)= Lv^{-1}$  for  all $u,v\in M$. 
Conditions  $(ii)$ and $(iii)$ in the theorem above express both that the Galois lattice $\gal(R)$ is finite.

%
%Indeed,
%Proposition \ref{prop:finitenesswords} is a special case of the fact that a language $F$ is
%rational (that is accepted by a finite state automaton) if and
%only if the minimal automaton  $\Min_F:= \{u^{-1}F :  u\in A^{*}\}$, is finite (Myhill-Nerode Theorem, see \cite{sakarovitch} Theorem 2.3, p.247). 
%Indeed, this last statement is equivalent to the fact  that the set $M$ 
%of residuals of the form $Fu^{-1}$, for $u\in A^{\ast }$ is finite. Since
%members of the injective envelope consist of intersections of these residuals, the
%finiteness of $M$ amounts
%to the finiteness of $\mathcal{S}_{F}$.\\

According to Higman's Theorem, if the  alphabet $A$ is w.q.o.,  $A^*$ equipped with the Higman ordering is w.q.o. Since this  ordering  is compatible with concatenation, we may apply $(v)$ to every final segment $F$ of $A^*$, obtaining that   the Galois lattice $\gal (R)$   is finite. Since the domain of this lattice is  $\mathcal{S}_{F}$, 
implication $(i) \Longrightarrow (ii)$ of Theorem \ref{thm:w.q.o.} follows.

%On the  other hand, $(ii) \Longrightarrow (i)$
%shows that  follows immediately from the
%rationality of $F$. However, notice that the finiteness of
%$\mathcal{S}_{F}$ can be derived  from Higman's
%Theorem without resorting to automata theory. Indeed, from Higman's Theorem,  if $A$ is finite then 
%$A^{\ast }$ is w.q.o. and  even more:  $\mathbf F\left( A^{\ast }\right)$
%is w.q.o. (a fact which does not necessarily hold if $A$ is infinite), hence $%
%\mathbf F ( A^{\ast }) \times \mathbf F( A^{\ast }) $ is w.q.o.
%It follows that the number of minimal elements of every subset  of
%$\mathbf F( A^{\ast }) \times \mathbf F( A^{\ast }) $ is
%finite. Since members
%of $\mathcal{S}_{F}$ consist of minimal pairs, it follows that $\mathcal{S}
%_{F}$ is finite. 

\section{Interval orders, Ferrers relations and injective envelope}\label{ferrers relation}

  We record the characterization of interval orders and Ferrers
relations. These two notions are intertwined.  Interval orders are
those orderings for which the irreflexive part is a Ferrers
relation. Ferrers relations have been introduced by J.Riguet
\cite{Ri}. Interval orders are studied by Fishburn in
\cite{Fi}. Part of the characterization of these
relations given below is due to Wiener \cite{Wi}.
Let $C$ be a chain; an \textit{interval} of $C$ is any subset of  $I$  such
that
$$x\in I,y\in I \ \mbox{and}\ \ x\leq z\leq y \ \mbox{imply}\ \ z\in I.   \eqno(1)
$$\
The collection  $Int$ $C$ of non-empty intervals of $C$ is ordered
as follows:
$$X<Y\;  \mbox{if}\; x<y \; \mbox{for all}\; \; x\in X, y\in Y.
\eqno( 2) $$

Let $P$ be a poset. The ordering on P, or P itself, is an {\it interval
order} if P is order isomorphic to a collection of non-empty
intervals of some chain $C$, ordered by condition $\left(
2\right).$\\
To each incidence structure $R:= \left( V,\rho,W\right) $ we
associate a poset $B(R):=(P, \leq )$ defined as follows:\\
The domain of $P$ is $V$ $\times \left\{ 0\right\} \cup W\times
\left\{ 1\right\},$ for $u=\left( x,i\right),v=\left( y,j\right)
\in P,$ the order relation is defined by:

  $$u<v\;  \mbox{if}\; \ i<j \; \mbox{and}\;  x \rho y. \eqno( 3) $$\\

Let $R:=(V,\rho,W)$ be an incidence structure. We say that $R$\ is
\textit{ Ferrers} or $\rho $ is a \textit{Ferrers relation} if $R$
satisfies one of the following conditions (see \cite{Fi}):

\begin{proposition}\label{prop:intervalorder}
\ Let $R:=(V,\rho,W)$ be an incidence structure. The following
conditions are equivalent: 
\begin{enumerate}[(i)]
\item The set $\left\{ R( 
x) :x\in V\right\}$ is totally ordered by
inclusion; 
\item  The set $\left\{ R^{-1}(
y) :y\in W\right\} $ is totally ordered by
inclusion; 

\item  The Galois lattice $\gal \left( R\right) $ is
totally ordered by inclusion; 

\item  $R$ has a coding into a chain;

\item The poset $B(R) $ does not embed
the direct sum $2\oplus \ 2$ of two copies of the $2$-element chain $2$; 
\item The ordering on $B(R)$ is an
interval order; 

\item  $x\rho y$ and $x^{\prime }\rho y^{\prime }$ imply $ 
x\rho y^{\prime }$ or $x^{\prime }\rho y$ for all $x,x^{\prime }\in V$,  $y,y^{\prime }\in W.$
\end{enumerate}
\end{proposition}

We say that a language $L$  is \textit{Ferrers} if the relation $%
\rho _{L}$ is Ferrers. 

According to condition $(vii)$ of Proposition \ref {prop:intervalorder}, this amounts to the following condition:
$$xx^{\prime }\in L\;  \mbox{and}\; yy^{\prime }\in L\;  \text{imply}\;  xy^{\prime }\in L\; 
\text{or}\;  yx^{\prime }\in L\;  \text{for all}\; x,x^{\prime },y,y^{\prime }\in
A^{\ast }.    \eqno(4) $$

 The study of Ferrers relations leads to the notion of Ferrers
dimension of a binary relation: the least number of Ferrers
relations whose intersection is this relation. The study of this
notion,  initiated by Bouchet in  his thesis \cite{Bo},  yields numerous
interesting results (e.g., see \cite{cogis, doignon}). This  suggest to look at the  same direction in the theory of languages. But, we may notice that contrarily to the case of relations, not every language is  Ferrers, or is an intersection of Ferrers  languages.  In fact, if $w$ is any  word,   $\rho_{\{w\}}$ is a Ferrers relation only if $w=\Box$ ( indeed, let $R:= (A^{*}, \rho_{\{w\}}, A^*)$,  then $\gal (R)= \{\{u\}: u\;  \text{prefix of}\;  w\}\cup\{\emptyset, A^*\}$, thus  if $w\not = \Box$, the Galois lattice has at least two incomparable elements, namely $\{w\}$ and $\{\Box\}$  plus a top and a bottom, thus it is not a chain). Thus, if $w\not = \Box$,  $\{w\}$ is not a Ferrers language and since it is a singleton,  it is not a union of Ferrers languages, hence its complement $A^{*} \setminus \{w\}$ is not an intersection of Ferrers languages. 

Since the complement of a Ferrers relation is Ferrers, the
complement of a Ferrers language is Ferrers (apply $(1)$ of Fact \ref{fact:complement}). The concatenation  of two Ferrers languages is not Ferrers in general. For a simple minded example, let $A:= \{a,b\}$, let $U:=  \{a^n: n\geq 2\}$, $U':= \{b^n: n\geq 2\}$. Each of these languages is Ferrers, but the concatenation $UU'$is not: 
let $x:= a^2b$,$y:=b$ and $x':= a$,$y':=ab^2$. Then $xy=x'y'= a^2b^2\in UU'$ but neither  
$xy'= a^2bab^2$ nor $x'y=ab$ belong to $UU'$.

Recall that for two words $u$ and $v$, $u$ is a \emph{prefix} of $v$ if $v=uw$ for some word $w$;  similarly, $u$ is a \emph{suffix} of $v$  if $v= wu$ for some word $w$. 

\begin{fact}\label{fact:product} Let $U$ and $U'$ be two subsets of $A^{\ast }$. If $U,U'$
are Ferrers and $U$ is a final segment for the prefix ordering  or $U'$ is a final segment for the suffix ordering then  the concatenation $UU'$ is Ferrers.
\end{fact}

\begin{proof}
Let $L:= UU'$, let $xx^{\prime }\in L$ and $yy^{\prime }\in L$. We prove that either $xy'$ or $x'y$ belong to $L$. There are four cases to consider; we only consider two, the others being similar.

\noindent{\bf Case 1.} $x=x_{1}x_{2},y=y_{1}y_{2}$ with $x_{1},y_{1}\in U$ and $%
x_{2}x^{\prime }\in U'$ and $y_{2}y^{\prime }\in U'.$ Since $U'$ is
Ferrers, either $x_{2}y^{\prime }\in U'$ or $y_{2}x^{\prime
}\in U'.$ In the former case $xy^{\prime }\in UU'$ whereas in the latter $%
yx^{\prime }\in UU'.$\\
{\bf Case 2.} $x=x_{1}x_{2},y^{\prime}=y_{1}^{\prime }y_{2}^{\prime }$ with $x_{1}\in U,$ $%
yy_{1}^{\prime }\in U,$ $x_{2}x^{\prime }\in U'$ and $y_{2}^{\prime }\in U'.$ If $U$ is a final segment for the prefix ordering, then  $x_1x_2y'_1\in U$ since $x_1$ belongs to $U$  and is a prefix  of $x_1x_2y'_1$. Thus  $xy'= x_1x_2 y'_1y'_2\in UU'$.   If $U'$ is a final segment for the suffix ordering, then $x_2y'_1y'_2\in U'$ since $y'_2$ belongs to $U'$  and is a suffix of $y'_2$. Thus $xy'= x_1x_2 y'_1y'_2\in UU'$.  \end{proof}

\begin{corollary} \label {productFerrers}
The concatenation of finitely many Ferrers final segments of $A^{*}$ is a Ferrers final segment. 
\end{corollary}

We recall that an \emph{ideal} of an ordered set $P$ is any non-empty initial segment $\mathcal I$ which is up-directed (that is any two elements $ x$ and $y$ of $\mathcal I$ have an upper bound $z$ in $\mathcal I$). Filters are defined dually. Ideals of $P$ are the join-irreducible elements of the lattice $I(P)$ of initial segments of $P$.  Ideals of the poset  $ A^{\ast}$ equipped with the Higman ordering have been described when $A$ is finite by Jullien \cite{jullien} and by us \cite{KP3} when $A$ is an ordered alphabet possibly infinite.  According to  Jullien, an  \emph{elementary ideal} of $A^{\ast}$ is any set  of the form $J�\cup \{\Box\}$  for some non empty ideal $J$  of $A$,  a \emph{star-ideal} is any set of the  form $I^{\ast}$  for some initial segment $I$ of $A$. Products of ideals are ideals. It is proved in \cite {KP3}  that every ideal  is a finite product of elementary and star-ideals if and only if the alphabet is well-quasi-ordered.

%We say that an age A is indecomposable if it is not the product of two ages distincts from A. We show that the indecomposable ages are exactly the elementary - and star-ages. Using the notion of indecomposability, we prove 

\begin{fact}\label{fact:ideal}  Ideals and filters of $A^{\ast}$ are Ferrers.
\end{fact}
\begin{proof}
Let $\mathcal I$ be an ideal of $A^{\ast }$. Suppose $xx^{\prime },$ $yy^{\prime
}\in \mathcal I$. We prove that $xy^{\prime}$ or $yx^{\prime }\in \mathcal I.$ Let $z\in \mathcal I$ such that $xx^{\prime },$ $yy^{\prime }\leq z.$ Let $%
z_{1}$ be the least prefix of $z$ such that $x,y\leq z_{1}$ and $z_{2}$ the
corresponding suffix, i.e.,  $z=z_{1}z_{2}.$ If  $x^{\prime}\leq z_{2}$ or $y^{\prime} \leq z_{2},$ then since $x, y
 \leq z_1$ then  $xy^{\prime }$ or $yx^{\prime }\in \mathcal I.$ If not, then since $xx'\leq z$ and $yy'\leq z$, we have $x,y\leq z^{-}_1$, where $z^{-}_{1}$ is obtained from $z_1$ by deleting its last letter. This contradicts the choice of $z_1$. Hence, $\mathcal I$ is Ferrers.   Let $\mathcal F$ be a filter of $A^{\ast}$. Let $xx^{\prime }$, $yy^{\prime
}\in \mathcal F$.  Let $z\in \mathcal F$ such that $z\leq  xx^{\prime }$ and  $z\leq yy^{\prime
}$. Write $z= z^{x}z^{x^{\prime}}$ with $z^{x}\leq x$, $z^{x^{\prime}}\leq x^{\prime}$ and $z= z^{y}z^{y^{\prime}}$ with $z^{y}\leq y$, $z^{y^{\prime}}\leq y^{\prime}$. Either $z^{y^{\prime}}\leq z^{x^{\prime}}$ or $z^{x^{\prime}}\leq z^{y^{\prime}}$. In the first case, since $z^{y}\leq y$ and $z^{y^{\prime}}\leq z^{x^{\prime}}\leq x^{\prime}$ we have $z= z^{y} z^{y'}\leq yx'$, hence $yx'\in \mathcal F$. In the second case we obtain similarly $xy'\in \mathcal F$. 
Hence $\mathcal F$ is Ferrers. \end{proof}

For every $u\in A^{*}$, the  initial segment $\downarrow u$ of $A^*$ is an ideal and the final segment $\uparrow u$ of $A^*$ is a filter, hence: 

\begin{corollary}\label{cor:initial-final ferrers} For every $u\in A^{*}$,   the  initial segment $\downarrow u$ of $A^*$ and the final segment $\uparrow u$ of $A^*$ are Ferrers.
\end{corollary}

Since $\{u\}= \downarrow u \bigcap \uparrow u$, we have:

\begin{corollary}\label{cor:intersection of two} For every $u\in A^{*}$, $\{u\}$ is an intersection of two Ferrers languages and   $A^{*}\setminus \{u\}$ is an union of two Ferrers languages.
\end{corollary}

If $A$ is w.q.o., every final segment is finitely generated and every ideal is a finite union of ideals, hence from the second part of Corollary \ref{cor:initial-final ferrers} and  from the first part of Fact \ref{fact:ideal}, we obtain.

\begin{corollary}\label{cor:wqo-ferrers} If $A$ is w.q.o., every final segment  if a finite union and a finite intersection of Ferrers languages. \end{corollary}

%p17 l-6 (je ne suis pas sur)  The aforementionned avec un seul n ou 2n?
%p20   corollary 20 for every l \in A*, the initial segment .....
%         Juste avant la proposition 6, on a besoin uniquement du corollaire 21.
%Corollary 22 If A is w.q.o.,, every final segment is a finite union of Ferrers languages and every ideal is a finite union of finite intersection of Ferrers languages.

From Corollary  \ref{productFerrers} and Fact \ref{fact:ideal} follows:

\begin{proposition}\label{prop:produit}Final segments of $A^{\ast }$
which are finite product of complement of ideals of $A^{\ast
}$  are Ferrers. 
\end{proposition}

According to Corollary \ref{cor:intersection of two}, for every $u\in A^{*}$, $A^*\setminus\{u\}$ is an union of two Ferrers languages, hence:

\begin{proposition}
Every language is  a union, possibly infinite,  of a family of intersections of two Ferrers languages. 
\end{proposition}

According to Corollary \ref{cor:wqo-ferrers}, if the alphabet is w.q.o. (and particularly, if it is finite), Boolean  combinations of final segments, alias piecewise testable languages,   are Boolean combination of  rational Ferrers languages.  

$\bullet$ \emph{If the alphabet $A$ consists of one letter, say $a$, these two Boolean algebras coincide}. Indeed, if $L$ is Ferrers then with respect to the natural order on $A^*:= \{a^n: n\in \N\}$, it is convex. Otherwise, there are $n<p<m\in \N$ such that $a^n, a^m\in L$, $a^p\not \in L$. Choosing $n,m$ with the difference $m-n$ minimum, we have $a^q\not \in L$ for $n<p<m$. Let $X:= (a^n)^{-1}L$ and $Y:=(a^{m-1})^{-1}L$. Then $\Box$ and $a^{m-n}\in X$ but $\Box \not \in X$, whereas $a\in Y$ but $\Box \not \in Y$. Hence, $X$ and $Y$ are incomparable with respect to inclusion, contradicting the fact that $L$ is Ferrers. Being convex, $L$ is the intersection of an initial segment segment with a final segment, thus it is piecewise testable. 

$\bullet$\emph{If $A:= \{a,b\}$, with $a\not =b$, then  $L:= A^*b$ is rational and Ferrers and not piecewise testable}.  Indeed, let $Q_L:= \{u^{-1}L: u\in A^*\}$, then $Q_L$ has two elements, namely $L$ and $L':= \{\Box\} \cup L$ (in fact $a^{-1} L=L$, $b^{-1}L=L'$, $a^{-1}L'=L$, $b^{-1}L'=L'$. The fact that $L$ is not piecewise testable follows from Stern'criterium (\cite{stern} Theorem 1.2): the sequence   $(u_n)_{n\in \N}$  defined by $u_{2n}:= (ba)^nb$ and $u_{2n+1}:= ( ba)^{n+1}$ is increasing for the subword ordering while $a_{2n} \in L$ and $a_{2n+1}\not \in L $ for all $n\in \N$. 

The language $L$ above has dot-depth one. Is this general? That is:

\begin{question}Do  rational Ferrers languages have dot-depth one? 
\end{question}

We relate Ferrers piecewise testable languages  and structural properties of transition systems. 

\begin{theorem}\label{prop:ferrerslanguage}
\ Let $F$ be a final segment of A$^{\ast }.$ The following
conditions are equivalent:
\begin{enumerate}
\item  $F$ is Ferrers; 
\item  The  space $\mathcal S_{F}$ is linearly
orderable.
\end{enumerate}
\end{theorem}

\begin{proof}
Proposition \ref{prop:intervalorder} and Proposition \ref{prop:linearly}.
\end{proof}

\begin{corollary}\label{prop:indec-linearly}
The finitely indecomposable absolute retracts are linearly orderable.
\end{corollary}

\begin{proof}
Let $E$ be a finitely indecomposable absolute retract. From Theorem \ref{thm:indec}, $E$
is isomorphic to $\mathcal{S}_{F}$ where $F$ is   join-irreducible in the lattice $\mathbf F(A^{\ast})$ ordered by reverse of inclusion. The fact that $F$ is   join-irreducible amounts to the fact that $A^{\ast
}\setminus F$ is an ideal of  $A^{\ast}$. According to Fact \ref{fact:ideal}, $A^{\ast
}\setminus F$ is Ferrers. Hence, its complement $F$ is Ferrers. The result follows from
Theorem  \ref{prop:ferrerslanguage}.
\end{proof}

%\begin{figure}[h]
%%\tabcapfont
%\centerline{%
%\begin{tabular}{c@{\hspace{1pc}}c}
%\includegraphics[width=15pc,height=15pc]{fig2left.eps} &
%\includegraphics[width=15pc,height=15pc]{fig2right.eps} \\
%&
%\end{tabular}}
%%\caption{\centerline{}}\label{f2lr}
%\end{figure}
%\centerline{fig.2}

\end{document}